\documentclass[11pt,reqno]{amsart}
\usepackage{amsfonts}
\usepackage{amsmath}
\allowdisplaybreaks[4]
\usepackage{amssymb,latexsym}
\usepackage[matrix,arrow]{xy}
\usepackage[usenames]{color}
\usepackage{slashed}
\usepackage{parskip}
\usepackage{amsmath,amssymb,graphics}
\usepackage[colorlinks=true,citecolor=blue]{hyperref}
\usepackage{graphicx}

\newcommand{\mathsym}[1]{{}}

\newtheorem{thm}{Theorem}[section]
\newtheorem{cor}[thm]{Corollary}
\newtheorem{lem}[thm]{Lemma}

\theoremstyle{definition}
\newtheorem{defn}{Definition}[section]
\numberwithin{equation}{section}
\theoremstyle{remark}
\newtheorem{rem}{Remark}[section]
\theoremstyle{example}

\newcommand{\n}{\nabla}

\newcommand{\ld}{\lambda}

\newcommand{\no }{\nonumber }
\newcommand{\intm }{\int_{M} }
\newcommand{\dmu }{\mathrm{d}\mu }
\newcommand{\md }{\mathrm{d}}

\newcommand{\mS }{\mathcal{S}}
\newcommand{\dmut}{\mathrm{d}\mu(t)}
\newcommand{\dmuo}{\mathrm{d}\mu(0)}
\newcommand{\de}{\partial}
\newcommand{\ov}{\overline}
\newcommand{\fa}{ \mathcal{F}^{+}}
\newcommand{\be}{\begin{equation}}
\newcommand{\ee}{\end{equation}}
\newcommand{\ba}{\begin{eqnarray}}
\newcommand{\ea}{\end{eqnarray}}
\newcommand{\ban}{\begin{eqnarray*}}
\newcommand{\ean}{\end{eqnarray*}}
\newcommand{\bag}{\begin{aligned}}
\newcommand{\eag}{\end{aligned}}
\newcommand{\lf}{\left}
\newcommand{\rt}{\right}

\newcommand{\al}{\alpha}
\newcommand{\suml}{\sum\limits }
\newcommand{\W}{\mathcal{W} }
\newcommand{\bpf}{\begin{proof} }
\newcommand{\epf}{\end{proof} }

\newcommand{\X}{\mathfrak{X}(M)}

\newcommand{\tr}{\mathrm{Tr}}
\newcommand{\F}{\mathcal{F}}
\newcommand{\ti}{\tilde}
\newcommand{\mn}{\sqrt{-1}}

\newcommand{\dist}{\mathrm{dist}}
\newcommand{\R}{ \mathbb{R}}
\newcommand{\Z}{ \mathbb{Z}}

\begin{document}
\title{the (logarithmic) Sobolev inequalities along geometric flow and applications}
\author{Shouwen Fang$^*$, Tao Zheng$^{**}$ }
{\small\thanks{$^{*}$Supported by National Natural Science Foundation of China grant Nos.11401514 and 11371310, and the University Science Research
Project of Jiangsu Province grant No. 13KJB110029.}
\thanks{$^{**}$Supported by the fundamental fund of Beijing Institute of Technology Nos.20131742009
and 20141742002, China Postdoctoral Science Foundation funded project grant Nos.2014M550620 and 2015T80040, and National Natural Science Foundation of China grant Nos.11401023 and 11471180.}}

\maketitle
\begin{abstract}
For some class of geometric flows, we obtain the (logarithmic) Sobolev inequalities and their equivalence up to  different factors directly and also obtain the long time non-collapsing and non-inflated properties, which generalize the results in the case of Ricci flow or List-Ricci flow or harmonic-Ricci flow. As applications, for mean curvature flow in Lorentzian space with nonnegative sectional curvature and twisted K\"{a}hler-Ricci flow on Fano manifolds, we get the results above.
\end{abstract}
\section{Introduction}

The role played by Sobolev inequality in analysis and geometry is well known and a fair amount of work has been devoted to its study.
Let $(M,\,g)$ be an $n$-dimensional ($n\geq 3$) compact Riemannian manifold.
Aubin\cite{aubin} proved the following Sobolev inequality
$$
\lf(\intm |f|^{\frac{2n}{n-2}}\dmu\rt)^{\frac{n-2}{n}}\leq \al\intm |\n f|^2\dmu+\beta\intm f^2\dmu,\quad\forall\, f\in \,W^{1,2}(M),
$$
where
$$
\al=[K(n)]^2+\varepsilon,\;\varepsilon>0
$$
and $\beta$ depends on bounds on the injectivity radius, sectional curvature and its derivatives and
$K(n)$ is the best constant in the Sobolev inequality for $\R^n$ (see \cite{talenti}).  Hebey \cite{hebey} proved that $\beta$ can depend only on $\varepsilon$, the injective radius and the lower bound of the Ricci curvature. Hebey and Vaugon \cite{hv} proved that we can take $\varepsilon=0$ but $\beta$ still depends on the derivatives of curvature tensor.

Assume that $Ric\geq -Kg$, where $K$ is a nonnegative constant. We consider Sobolev inequality like
\ba\label{sobsharp}
\lf(\intm |f-f_M|^{\frac{2n}{n-2}}\dmu\rt)^{\frac{n-2}{n}}\leq S(M)\intm |\n f|^2\dmu,\quad \forall\;f\in C^{\infty}(M,\,\R).
\ea
Gallot \cite{gallot} proved
\be\label{sobxi}
S(M)\leq e^{C_n(1+\sqrt{K}\mathrm{diam}(M))}[\mathrm{diam}(M)]^2[\mathrm{Vol}_{g}(M)]^{-\frac{2}{n}}.
\ee
Apart form the dimensional constant, the estimate above is sharp.

Let $B:=B(x,\,r)\subset M$ be a ball with center $x$ and radius $r$. Then in view of (\ref{sobsharp}) and (\ref{sobxi}), it is natural to conjecture that
$$
\lf(\int_{B } |f-f_{B }|^{\frac{2n}{n-2}}\dmu\rt)^{\frac{n-2}{n}}\leq e^{C_n(1+\sqrt{K}r)} r^2[\mathrm{Vol}_{g}(B)]^{-\frac{2}{n}}
\int_B |\n f|^2\dmu,\quad \forall\;f\in C^{\infty}(B,\,\R),
$$
where $K$ is a nonnegative constant such that
$$
Ric\geq -Kg,\quad \mbox{on}\;B(x,\,2r).
$$
Saloff-Coste \cite{saloff} solved the conjecture partially. They proved that, for any $f\in C^{\infty}_{0}(B,\,\R)$, if $n\geq 3$, there holds
\ba
\lf(\int_B |f|^{\frac{2n}{n-2}}\dmu\rt)^{\frac{n-2}{n}}\leq e^{C_n(1+\sqrt{K}r)} r^2[\mathrm{Vol}_{g}(B)]^{-\frac{2}{n}}
\int_B \lf(|\n f|^2 +r^{-2} f^2\rt)\dmu\no
\ea
and if $n\leq 2$,  the above inequality holds with $n$  replaced by any fixed $n'>2$.
More details about Sobolev inequality can be found in Aubin and Li \cite{aubinli}, Biezuner \cite{biezuner} and the references therein.

In the case of Ricci flow
\ba\label{rf}
\frac{\de}{\de t}g_{ij}(x,t)=-2R_{ij}(x,t),
\ea
(logarithmic) Sobolev inequalities also play an important role in its analysis. One motivation for the $\W$-entropy comes from the log-Sobolev inequality of Gross \cite{gross}(see also Topping \cite{topping}). Due to the importance of (logarithmic) Sobolev inequality in the analysis of geometric flow, it is key to have a uniform control on the constants $\al$ and $\beta$.

\v{S}e\v{s}um and Tian \cite{sesumtian} proved a uniform Sobolev imbedding for certain K\"{a}hler-Ricci flow with Ricci curvature bounded from below.

However, in general the constant $\beta$ cannot be controlled uniformly along the Ricci flow. By making use of the (generalized) Perelman's $\W$ entropy \cite{perelman}, Zhang \cite{zhangqi1,zhangqi2} and  Ye \cite{ye2,ye4,ye3,ye1}(see also Hsu\cite{hsu}) proved (logarithmic) Sobolev inequalities along Ricci flow, from which and the method of \cite{carron} (see also Lemma 2.2 in \cite{hebey2} and its proof or Lemma 6.1 in \cite{ye1}) they established long time non-collapsing result generalizing the Perelman's short time result \cite{perelman}. Zhang \cite{zhangqi3} also proved the long time non-inflated result for the normalized K\"{a}hler-Ricci flow on Fano manifolds.

In this paper, we consider the geometric flow
\ba\label{grf}
\frac{\de}{\de t}g_{ij}(x,t)=-2\mS_{ij}(x,t)
\ea
on $M\times[0,\,T)$ for some (finite or infinite) $T>0$ with a given initial Riemannian metric $g(0)=g_{0}$, where $\mS_{ij}(x,t)$ is the component of a time-dependent symmetric $2$-tensor $\mS$.
Motivated by \cite{perelman}, we define the $\F$ functional and $\W$ functional and prove their monotonicity under some assumptions. Next we obtain the (logarithmic) Sobolev inequalities and their equivalence up to  different factors. We also prove the long time non-collapsing and non-inflated. As applications, for mean curvature flow in Lorentzian space and twisted K\"{a}hler-Ricci flow on Fano manifolds, we get the results above.

In the following, we denote
the volume element of $g(t)$ by $\dmut$, the trace of $\mS_{ij}(t)$ by $S_{t}\;\mbox{or}\;S(x,t)=\suml_{i,j=1}^n g^{ij}(t)\mS_{ij}(t)$ (sometimes also by $S$ simply without confusion),
the volume of $M$ with respect  to $g(t)$ by $\mathrm{Vol}_{g(t)}(M)$,
the first eigenvalue of $-\Delta_{g_{t}}+\frac{S_{t}}{4}$ by $\ld_{0}(g(t))$
and the norm of the gradient of $u\in W^{1,2}(M)$ with respect  to $g(t)$ by $|\n u|_{t}$.

For convenience, we define an evolving tensor quantity $\mathcal{D}_2$ associated to the tensor $\mS$ (see for example \cite{fangzhu} and the references therein).
\begin{defn}
Let $g(x,t)$ be a smooth solution to the geometric flow (\ref{grf}) on $M\times[0,T)$. Then for any $X\in \X$, we define
\ba\label{con}
\mathcal {D}_2( \mathcal{S},X)&:=&\frac{\partial S}{\partial
t}-\Delta_{g(t)} S-2|\mathcal{S}|_{g(t)}^2\no\\
&&+4(\nabla^i\mS_{ij})X^j-2X^i\nabla_iS+2R_{ij}X^iX^j-2\mathcal{S}_{ij}X^iX^j,
\ea
where $\n$ and $R_{ij}$ are the Levi-Civita connection and Ricci curvature respectively with respect  to the Riemannian metric $g(t)$.
In particular, if for any vector field $X\in \X$ there holds $\mathcal {D}_2( \mathcal{S},X)\geq0$ on $[0,\,T)$, then we call $\mathcal {D}_2( \mathcal{S},\cdot)$ nonnegative.
\end{defn}
\begin{thm}\label{thmlogsob1}
Assume that $g(x,t)$ is a smooth solution to the geometric flow (\ref{grf}) in $M\times[0,T)$ and that $\mathcal{D}_{2}(\mS,\cdot)$ defined in (\ref{con}) is nonnegative.
For each $\sigma>0$ and each $t\in[0,\,T)$, we have
\ba\label{logsob1}
\int_{M}u^2\ln u^2 \dmut\leq \sigma \int_{M}\lf(|\n u|_{t}^2+\frac{S_{t}}{4}u^2\rt)\dmut-\frac{n}{2}\ln \sigma
+A_1\lf(t+\frac{\sigma}{4}\rt)+A_2
\ea
  for any $u\in W^{1,2}(M)$ with  $\int_M u^2 \dmut=1$,
where
\ba
A_1&=&\frac{4}{  C_S(M, g_0)^2\mathrm{Vol}_{g_0}(M)^{\frac{2}{n}}}-\min S_{0}, \label{a1} \\
A_2&=& n\ln   C_S(M,g_0)+
\frac{n}{2}(\ln n-1), \label{a2}
\ea
and
where $ C_S(M,g_0)$ is the Sobolev constant defined in ( \ref{sob}).

Therefore, we can deduce
\begin{align} \label{stronglogsob1}
\int_M u^2 \ln u^2 \dmut \le \frac{n}{2} \ln \left[ \alpha_{\MakeUppercase{\romannumeral1}}\lf\{\int_M \lf(|\nabla u|_{t}^2 +\frac{S_{t}}{4} u^2\rt)\dmut+\frac{A_1}{4}\rt\} \right]
\end{align}
for any $u \in W^{1,2}(M)$ satisfying $\int_M u^2\dmut=1$,
where
$$
\alpha_{\MakeUppercase{\romannumeral1}}=\frac{2e}{n}e^{\frac{2(A_1t+A_2)}{n}}.
$$
\end{thm}

\begin{thm}\label{thmlogsob3}
Assume that $g(x,t)$ is a smooth solution to the geometric flow (\ref{grf}) in $M\times[0,T)$ and that $\mathcal{D}_{2}(\mS,\cdot)$ defined in (\ref{con}) is nonnegative. If $\ld_{0}(g_0)$ is positive, then for each $t\in [0, T)$ and each $\sigma>0$ there holds
\begin{align} \label{logsob3}
\int_M u^2 \ln u^2 \dmut \le \sigma \int_M \lf(|\nabla u|_{t}^2 +\frac{S_{t}}{4}u^2\rt) \dmut
-\frac{n}{2}\ln \sigma +C
\end{align}
for all $u \in W^{1,2}(M)$ with $\int_M u^2 \dmut=1$, where $C$ depends only on the dimension $n$, $\mathrm{Vol}_{g_0}(M)$, $C_S(M, g_0)$, $\lambda_0(g_0)$ and the lower bound for $S_{0}$.

Therefore, there holds for each $t \in [0, T)$
\begin{align} \label{stronglogsob3}
\int_M u^2 \ln u^2 \dmut \le \frac{n}{2} \ln \left[ \alpha_{\MakeUppercase{\romannumeral2}}\int_M \lf(|\nabla u|_t^2 +\frac{S_t}{4} u^2\rt)\dmut \right]
\end{align}
for all $u \in W^{1,2}(M)$ with $\int_M u^2 \dmut=1$,
where
$$
\alpha_{\MakeUppercase{\romannumeral2}}=\frac{2e}{n}e^{\frac{2 C}{n}}.
$$
\end{thm}
\begin{thm}\label{thmsob}
Assume that $g(x,t)$ is a smooth solution to the geometric flow (\ref{grf}) in $M\times[0,T)$ and that $\mathcal{D}_{2}(\mS,\cdot)$ defined in (\ref{con}) is nonnegative. There hold
\begin{enumerate}
\item if $\ld_{0}(g_{0})>0$, for $t\in [0,\,T)$ and $u\in W^{1,2}(M)$, there holds
\begin{align}\label{sobp}
\lf(\intm |u|^{\frac{2n}{n-2}}\dmut\rt)^{\frac{n-2}{n}}\leq A\intm\lf( |\n u |_{t}^2+\frac{S_t}{4}u^2\rt)\dmut,
\end{align}
where $A$ is a positive defined in (\ref{sobpc}), depending only on $n,\,\mathrm{Vol}_{g_0}(M),\,C_{S}(M,\,g_0),\,\ld_0(g_0)$ and the lower bound of $S_0$.
\item if $T<\infty$, for $t\in [0,\,T)$ and $u\in W^{1,2}(M)$, there holds
\begin{align}\label{sobi}
\lf(\intm |u|^{\frac{2n}{n-2}}\dmut\rt)^{\frac{n-2}{n}}\leq A\intm\lf( |\n u |_{t}^2+\frac{S_t}{4}u^2\rt)\dmut+
B\intm u^2\dmut,
\end{align}
where $A$ and $B$ are defined in (\ref{sobia}) and (\ref{sobib}) respectively, depending only on  $n$, $\mathrm{Vol}_{g_0}(M)$, $C_{S}(M,\,g_0)$, $\ld_0(g_0)$ and the upper bound   $T$.
\end{enumerate}
\end{thm}
\begin{rem}
From the Jensen's inequality, we can deduce logarithmic Sobolev inequality from Sobolev inequality(see for example). Now from the proof of Theorem \ref{thmsob}, we know that logarithmic Sobolev inequality implies also Sobolev inequality with  different  constants. Therefore, we can say that (logarithmic) Sobolev inequalities are equivalent to each other up to constant factors. In the case of Ricci flow (\ref{rf}), the equivalence proved by making use of estimates on heat kernel can be found in Ye \cite{ye1} and Zhang \cite{zhangqi2}.
\end{rem}
\begin{rem}
In the case of Ricci flow (\ref{rf}),  the results in Theorem \ref{thmlogsob1},
Theorem \ref{thmlogsob3} and Theorem \ref{thmsob} can be found in Zhang \cite{zhangqi1,zhangqi2}, Ye \cite{ye1} and Hsu \cite{hsu}.
\end{rem}
\begin{rem}
In the case of extended Ricci flow (so-called List-Ricci flow \cite{list})
$$
\lf\{
\bag
\frac{\de}{\de t}g_{ij}(x,t)=&-2R_{ij}(x,t)+4\md \phi(x,\,t)\otimes \md \phi(x,\,t),\\
\frac{\de}{\de t}\phi(x,\,t)=&\Delta_{g(x,\,t)}\phi(x,\,t),
\eag
\rt.
$$
where $\phi\in C^{\infty}(M\times\R,\,\R)$, the Sobolev inequalities were obtained by Liu and Wang \cite{liuwang}.
\end{rem}
\begin{rem}
In the case of harmonic-Ricci flow (see \cite{MH,MR,Zh})
$$
\lf\{
\bag
\frac{\de}{\de t}g_{ij}(x,t)=&-2R_{ij}(x,t)+2\alpha(t) \nabla\psi\otimes \nabla\psi ,\\
\frac{\de}{\de t} \psi(x,t)=&\tau_{g(x,t)}\psi(x,t),
\eag
\rt.
$$
where $\psi(\cdot,t):(M,g(\cdot,t))\rightarrow(N,h)$ is a family of smooth maps between two Riemannian manifolds, both $g(\cdot,t)$ and $h$ are Riemannian metrics, $\alpha(t)$ is a positive non-increasing function, and $\tau_{g}\psi$ denotes the intrinsic Laplacian of $\psi$,  the Sobolev inequalities can be found in \cite{fangzheng}.
\end{rem}
Given the (logarithmic) Sobolev inequalities, we can prove the $\kappa$-noncollapsing property and the so-called $\kappa$-noninflated property and also give
some examples as applications.

The rest of the paper is organized as follows.
In Section \ref{logsobrm}, we prove the equivalence between Sobolev inequality and logarithmic Sobolev inequality up to a different factor, which also holds in the case of geometric flow (\ref{grf}). In Section \ref{pre}, we define $\F$ functional and $\W$ entropy and prove their monotonicity and   prove the lower bound of $S$, assuming that $\mathcal{D}_2(\mS,\cdot)$ is nonnegative. In Section \ref{logsob}, we prove (logarithmic) Sobolev inequalities under geometric flow (\ref{grf}), i.e., Theorem \ref{thmlogsob1}, Theorem \ref{thmlogsob3} and Theorem \ref{thmsob}.
In Section \ref{noncollapsing}, we give the $\kappa$-noncollapsing property along geometric flow (\ref{grf}). In Section \ref{sectionnoninflated}, based on a series of properties of fundamental solution to conjugate heat equation along geometric flow (\ref{grf}), we prove the so-called $\kappa$ noninflated property. In Section \ref{app}, as applications, we consider  Lorentzian mean curvature flow (\ref{mcf}) on ambient Lorentzian manifold with nonnegative sectional curvature and twisted K\"{a}hler-Ricci flow (\ref{tkrf}) on Fano manifolds and obtain the results mentioned in the first six sections along these two geometric flows.
\section{The (Logarithmic) Sobolev Inequalities on Riemannian Manifolds and their relations}\label{logsobrm}
In this section, first we give some (logarithmic) Sobolev inequalities and lemmas which will be useful in the following sections.

Let $(M,\,g)$ be an $n$-dimensional ($n\geq 3$) compact Riemannian manifold.
Then the Sobolev constant of $(M,\,g)$ (for the exponent 2) is defined to be
\be\label{sob}
C_{S}(M,\,g)=\sup\lf\{\|u\|_{\frac{2n}{n-2}}-\frac{1}{\mathrm{Vol}_{g}(M)^{\frac{1}{n}}}\|u\|_{2}:\quad u\in C^{1}(M),\,\|\n u\|_{2}=1\rt\}.
\ee
Therefore, the Sobolev inequality (for the exponent 2) is
\be\label{sobolevine}
\|u\|_{\frac{2n}{n-2}}\leq C_{S}(M,\,g)\|\n u\|_{2}+\frac{1}{\mathrm{Vol}_{g}(M)^{\frac{1}{n}}}\|u\|_{2},\quad \forall\;u\in W^{1,2}(M).
\ee
We need the following fundamental results (see for example \cite {ye1}).
\begin{thm}
Let $(M,\,g)$ be an $n$-dimensional ($n\geq 3$) compact Riemannian manifold and $\mS$ be any symmetric $2$-tensor with trace $S=\suml_{i,j=1}^ng^{ij}\mS_{ij}$. Then for any $u\in W^{1,2}(M)$ with $\|u\|_2=1$, there hold
\ba
\int_{M}u^2\ln u^2\dmu&\leq& n\ln\lf(C_{S}(M,\,g)\|\n u\|_{2}+\frac{1}{\mathrm{Vol}_{g} (M)^{\frac{1}{n}}}\rt),\\
\int_M u^2 \ln u^2\dmu&\leq& \frac{n\alpha C_S(M,g)^2}{2} \int_M \lf(|\nabla u|^2+\frac{S}{4}u^2\rt) \dmu
-\frac{n}{2}\lf(\ln \alpha-\ln 2+1\rt) \nonumber \\
& & +\frac{n\alpha}{2}\lf(\frac{1}{\mathrm{Vol}_{g} (M)^{\frac{2}{n}}} -\frac{\min S^-}{4}C_S(M,g)^2\rt),\label{logsobr1}
\ea
where $\al$ is any positive real number and $S^-=\min\{S,\,0\}$.

Moreover, if the first eigenvalue $\ld_0 =\lambda_0(g)$ of the operator
$-\Delta_{g}+\frac{S}{4}$ is positive, we can deduce
\ba \label{RLS4}
\int_M u^2\ln u^2\dmu \le  \frac{nAC_S(M,\,g)^2}{2}\int_M\lf(|\nabla u|^2+\frac{S}{4}u^2\rt)
-\frac{n}{2}\ln A+\frac{n}{2}\ln 2+\sigma_0,
\ea
where

\ba \label{delta-0}
\delta_0=\delta_0(g)=
\lf(\lambda_0 C_S(M,\,g)^2+\frac{1}{ \mathrm{Vol}_{g}  (M)^{\frac{2}{n}}}
-C_S(M,\,g)^2 \frac{\min S^-}{4}\rt)^{-1},
\ea
\ba \label{sigma-0}
\sigma_0=\sigma_0(g)&= &\frac{n}{2}\Bigg[\ln \lf(\lambda_0 C_S(M,\,g)^2+\frac{1}{ \mathrm{Vol}_g(M)^{\frac{2}{n}}}
-C_S(M,\,g)^2 \frac{\min S^-}{4}\rt)\nonumber \\
&&  \quad\quad-\ln (\lambda_0 C_S(M,\,g)^2)-1\Bigg]
\ea
and $A$ is any positive real number satisfying $A \ge \delta_0$.
\end{thm}
Now we give some fundamental materials which will be useful in the proof of Logarithmic Sobolev inequality implying Sobolev inequality.
The ideas come from \cite{bakry} and the references therein.

Let $(M,\mathcal{E},\mu)$ be a measurable space with a nonnegative $\sigma$-finite measure $\mu$.
For convenience, let $\fa$ be nonnegative function on $M$ and be contained in all $L^p$-space with respect  to the measure $\mu$.

Let $W(f)$ be a given norm or semi-norm on $\fa$ which will be determined later. For $\rho>1,\,k\in \Z$, define
$$
f_{\rho,\,k}=\min\{(f-\rho^{k})^+,\,\rho^{k}(\rho-1)\},
$$
where $(f-\rho^{k})^+=\max\{f-\rho^{k},\,0\}$.

For any $f\in \fa$, define
$$
a_{f,p,k,\rho}=\rho^{pk}\mu(f\geq \rho^{k}).
$$
\begin{lem}
For any $f\in\fa$ and any $\rho>1$, we have
\ba\label{fpguanxi}
\frac{\rho^{p}-1}{\rho^{p}}\sum\limits_{k\in\Z}a_{f,p,k,\rho}\leq\|f\|_{p}^p\leq (\rho^{p}-1)\sum\limits_{k\in\Z}a_{f,p,k,\rho}.
\ea
\end{lem}
\begin{proof}
From
\ba
\intm f^{p}\dmu
&=&\sum\limits_{k\in\Z}\int_{\rho^{k}\leq f\leq \rho^{k+1}}f^p\dmu\no\\
&\leq&\sum\limits_{k\in\Z}\rho^{p(k+1)}\bigg(\mu(f\geq \rho^{k})-\mu(f\geq \rho^{k+1})\bigg)\no\\
&=&\rho^{p}\sum\limits_{k\in\Z}a_{f,p,k,\rho}-\sum\limits_{k\in\Z}a_{f,p,k+1,\rho}\no\\
&=&(\rho^{p}-1)\sum\limits_{k\in\Z}a_{f,p,k,\rho}\no
\ea
and
\ba
\intm f^{p}\dmu
&=&\sum\limits_{k\in\Z}\int_{\rho^{k}\leq f\leq \rho^{k+1}}f^p\dmu\no\\
&\geq&\sum\limits_{k\in\Z}\rho^{pk}\bigg(\mu(f\geq \rho^{k})-\mu(f\geq \rho^{k+1})\bigg)\no\\
&=&\sum\limits_{k\in\Z}a_{f,p,k,\rho}-\frac{1}{\rho^{p}}\sum\limits_{k\in\Z}a_{f,p,k+1,\rho}\no\\
&=&\frac{\rho^{p}-1}{\rho^{p}}\sum\limits_{k\in\Z}a_{f,p,k,\rho},\no
\ea
we can deduce (\ref{fpguanxi}).
\end{proof}
\begin{lem}\label{lem23}
For $f\in \fa$ and $1\leq p\leq +\infty$, we have
\ba\label{fguanxi2}
\lf(\frac{\rho-1}{\rho}\rt)^{p}\frac{1}{\rho^p-1}\|f\|_{p}^p\leq\sum\limits_{k\in\Z}\|f_{\rho,\,k}\|_{p}^p\leq \lf(\frac{\rho-1}{\rho}\rt)^{p-1}\|f\|_{p}^p.
\ea
\end{lem}
\begin{proof}
Since
\ba
\int_{M}|f_{\rho,\,k}|^{p}\dmu
&=&p\int_{0}^{\rho^{k+1}-\rho^{k}} t^{p-1}\mu(f-\rho^{k}\geq t)\md t\no\\
&=&p\int_{\rho^{k}}^{\rho^{k+1}}(s-\rho^{k})^{p-1}\mu(f\geq s)\md s,\no
\ea
for $p\geq 1$, we have
\ba
\sum\limits_{k\in\Z}\int_{M}|f_{\rho,\,k}|^{p}\dmu
&=&p\sum\limits_{k\in\Z}\int_{\rho^{k}}^{\rho^{k+1}}(s-\rho^{k})^{p-1}\mu(f\geq s)\md s\no\\
&\leq&\lf\{\sup\limits_{k\in\Z}\sup\limits_{s\in[\rho^{k},\,\rho^{k+1}]}\lf(\frac{ s-\rho^{k} }{s}\rt)^{p-1}\rt\}
\lf\{p\sum\limits_{k\in\Z}\int_{\rho^{k}}^{\rho^{k+1}}s^{p-1}\mu(f\geq s)\md s\rt\}\no\\
&= &\lf(\frac{ \rho-1 }{\rho}\rt)^{p-1}\int_{M} f ^p\dmu.\no
\ea
On the other hand, we have
\ba
\sum\limits_{k\in\Z}\int_{M}|f_{\rho,\,k}|^{p}\dmu
&=&p\sum\limits_{k\in\Z}\int_{\rho^{k}}^{\rho^{k+1}}(s-\rho^{k})^{p-1}\mu(f\geq s)\md s\no\\
&\geq&\sum\limits_{k\in\Z}\mu(f\geq \rho^{k+1})p\int_{\rho^{k}}^{\rho^{k+1}}(s-\rho^{k})^{p-1} \md s\no\\
&= &\lf(\frac{\rho-1}{\rho}\rt)^{p}\sum\limits_{k\in\Z}a_{f,p,k+1,\rho}\no\\
&\geq&\lf(\frac{\rho-1}{\rho}\rt)^{p}\frac{1}{\rho^p-1}\|f\|_{p}^p.\no
\ea
Thus, we can obtain (\ref{fguanxi2}).
\end{proof}
For $p,s \in (0,\,+\infty]$ and $\vartheta \in (0,\,1]$, assume that there holds
$$
\|f\|_{p}\leq (CW(f))^{\vartheta}\|f\|_{s}^{1-\vartheta}, \eqno{(S_{p,s}^{\vartheta})}
$$
where the associated parameter $q\in (-\infty,\,0)\cup (0,\,+\infty)\cup\{\infty\}$
by setting
\be\label{q}
\frac{1}{p}=\frac{\vartheta}{q}+\frac{1-\vartheta}{s}.
\ee
\begin{lem} \label{fao}
For a function $f\in \fa$, define
$$
\varphi:\;u\longmapsto \ln \|f\|_{\frac{1}{u}}.
$$
Then $\varphi''(u)\geq0$.
\end{lem}
\begin{proof}
Noting that
\ba
\varphi'(u)=-\|f\|_{\frac{1}{u}}^{-\frac{1}{u}}\intm f^{\frac{1}{u}}\ln\lf(\frac{f}{\|f\|_{\frac{1}{u}}}\rt)^{\frac{1}{u}},\no
\ea
for convenience, we define
$$
\phi(r):=-\|f\|_{r}^{-r}\intm f^{r}\ln\lf(\frac{f}{\|f\|_{r}}\rt)^{r}.
$$
Then we have
\ba
\phi'(r)&=&\frac{r}{\|f\|_{r}^{2r}}\lf\{\lf[\intm f^r\ln f\dmu\rt]^2-\lf(\intm f^r\dmu\rt)\lf(\intm f^r(\ln f)^2\dmu\rt)\rt\}\no\\
&\leq&\frac{r}{\|f\|_{r}^{2r}}\Bigg\{\lf[\intm \lf(f^{\frac{r}{2}}\rt)^2 \dmu\rt]
\lf[\intm \lf(f^{\frac{r}{2}}\ln f\rt)^2\dmu\rt]\no\\
&& \quad\quad\quad\quad-\lf(\intm f^r\dmu\rt)\lf(\intm f^r(\ln f)^2\dmu\rt)\Bigg\}=0.\no
\ea
Thus
$$
\varphi''(u)=-\frac{1}{u^2}\phi'\lf(\frac{1}{u}\rt)\geq0.
$$
\end{proof}
\begin{thm}[See Theorem 10.2 in \cite{bakry}]
If for any $f\in\fa$, we have logarithmic Sobolev inequality
$$
\intm\lf[f^p\ln\lf(\frac{f}{\|f\|_p}\rt)^p\dmu\rt]\leq \lf(\frac{1}{p}-\frac{1}{q}\rt)^{-1}\|f\|_{p}^p\ln\lf(\frac{CW(f)}{\|f\|_{p}}\rt)\eqno{(LS_{p}^q)},
$$
then we can deduce ($S_{p,s}^{\vartheta}$) for all $0<s<p$ and vise versa.
\end{thm}
\begin{proof}
From Lemma \ref{fao}, the function
$$
\psi(u)=\frac{\varphi(u)-\varphi\lf(\frac{1}{p}\rt)}{u-\frac{1}{p}}
$$
is increasing of $u$, where we can define $\psi\lf(\frac{1}{p}\rt)=\varphi'\lf(\frac{1}{p}\rt)$.

Therefore, from ($LS_{p}^q$), for $0<s<p$ we can deduce (noticing that $\frac{1}{p}>\frac{1}{q}$)
\ba
-\psi(s)&\leq &\psi\lf(\frac{1}{p}\rt)=\varphi'\lf(\frac{1}{p}\rt)\no\\
&=&\|f\|_{p}^{-p}\intm f^{p}\ln\lf(\frac{f}{\|f\|_{p}}\rt)^{p}\no\\
&\leq& \lf(\frac{1}{p}-\frac{1}{q}\rt)^{-1} \ln\lf(\frac{CW(f)}{\|f\|_{p}}\rt) ,\no
\ea
which is   ($S_{p,s}^{\vartheta}$) exactly.

Now assume ($S_{p,s}^{\vartheta}$) holds for any $0<s<p$. Rewrite ($S_{p,q}^{\vartheta}$) as
$$
\lf(\frac{\|f\|_p}{\|f\|_s}\rt)^{\lf(\frac{1}{s}-\frac{1}{p}\rt)^{-1}}\leq\lf(\frac{CW(f)}{\|f\|_s}\rt)^{\lf(\frac{1}{s}-\frac{1}{q}\rt)^{-1}}.
$$
Taking logarithms, we have
$$
\lf(\ln\|f\|_p-\ln\|f\|_s\rt)\lf(\frac{1}{s}-\frac{1}{p}\rt)^{-1}\leq \lf(\frac{1}{s}-\frac{1}{q}\rt)^{-1}
\ln\lf(\frac{CW(f)}{\|f\|_{s}}\rt).
$$
Letting $s\longrightarrow p$, we get ($LS_{p}^q$).
\end{proof}
\begin{lem}\label{lem6}
If for $\al>0$, there holds
$$
\lf(\sum\limits_{k\in\Z}W(f_{\rho,k})^{\al}\rt)^{\frac{1}{\al}}\leq A(\al,\rho)W(f),\quad \forall\;f\in\fa,
$$
where $A(\al,\rho)$ is a constant depending on $\al$ and $\rho$, then ($S_{p,s}^{\vartheta}$) implies
$$
\|f\|_q\leq  (\rho^{q}-1)^{\frac{1}{q}}\rho^{\frac{q-s}{p-s}}\frac{CA(p,\rho)}{\rho-1} W(f).\eqno{(S_{q,p,s})}
$$
\end{lem}
\begin{proof}
From ($S_{p,s}^{\vartheta}$), we have
\ba\label{srsv1}
\int_{M} f_{\rho,k}^{p}\dmu\leq (CW(f_{\rho,k}))^{p\vartheta }\lf(\intm f_{\rho,k}^{s}\dmu\rt)^{\frac{p(1-\vartheta)}{s}}.
\ea
Since
\ba\label{srsv2}
\intm f_{\rho,k}^{s}\dmu
&\leq& \rho^{sk}(\rho-1)^s\mu(f\geq \rho^{k})\no\\
\int_{M} f_{\rho,k}^{p}\dmu
&\geq& \rho^{pk}(\rho-1)^p\mu(f\geq \rho^{k+1}),
\ea
we can deduce
\ba
a_{f,q,k+1,\rho}\leq \rho^{q}(\rho-1)^{-p\vartheta}(CW(f_{\rho,k}))^{p\vartheta }a_{f,q,k,\rho}^{\frac{p(1-\vartheta)}{s}}.\no
\ea
Therefore, we have
\ba
\sum\limits_{k\in\Z}a_{f,q,k,\rho}
&=&\sum\limits_{k\in\Z}a_{f,q,k+1,\rho}\no\\
&\leq&\sum\limits_{k\in\Z}\rho^{q}(\rho-1)^{-p\vartheta}(CW(f_{\rho,k}))^{p\vartheta }a_{f,q,k,\rho}^{\frac{p(1-\vartheta)}{s}}\no\\
&\leq&\rho^{q}(\rho-1)^{-p\vartheta}C^{p\vartheta }
\lf(\sum\limits_{k\in\Z}W(f_{\rho,k})^{p}\rt)^{\vartheta }
\lf(\sum\limits_{k\in\Z}a_{f,q,k,\rho}^{\frac{p }{s}}\rt)^{1-\vartheta }\no\\
&\leq&\rho^{q}(\rho-1)^{-p\vartheta}C^{p\vartheta }
\lf(\sum\limits_{k\in\Z}W(f_{\rho,k})^{p}\rt)^{\vartheta }
\lf(\sum\limits_{k\in\Z}a_{f,q,k,\rho}\rt)^{\frac{p(1-\vartheta)}{s}}.\no
\ea
Therefore, we have
\ba\label{yong}
\sum\limits_{k\in\Z}a_{f,q,k,\rho}
&\leq &\rho^{q\frac{q-s}{p-s}}(\rho-1)^{-q}C^{q}\lf(\sum\limits_{k\in\Z}W(f_{\rho,k})^{p}\rt)^{\frac{q}{p} }.
\ea
Taking $p=q$ in (\ref{fpguanxi}), from (\ref{yong}), we can deduce
\ba
\intm f^{q}\dmu
&\leq&(\rho^{q}-1)\rho^{q\frac{q-s}{p-s}}(\rho-1)^{-q}C^{q}\lf(\sum\limits_{k\in\Z}W(f_{\rho,k})^{p}\rt)^{\frac{q}{p} }\no\\
&\leq&(\rho^{q}-1)\rho^{q\frac{q-s}{p-s}}(\rho-1)^{-q}C^{q}A(p,\rho)^qW(f)^{q},\no
\ea
which is ($S_{q,p,s}$) as desired.
\end{proof}
Let $(M,\,g)$ be an $n$-dimensional Riemannian manifold. Then for $  f  \in \fa$, define non-negative functional
$$
W(f)=\lf(\int_{M}(|\n f|^p+S f ^{p})\md \mu +c\intm f^p\md\mu\rt)^{\frac{1}{p}},
$$
where
$S\in C^{0}(M,\,\R)$ and $c$ is a constant.
\begin{lem}\label{lem7}
If $c+S\geq0$ and $1\leq p<+\infty$, then we have
\ba
\lf(\sum\limits_{k\in\Z}W(f_{\rho,k})^{\al}\rt)^{\frac{1}{\al}}\no
\leq W(f),
\ea
where any $\al\geq p$ is   constant.
\end{lem}
\begin{proof}
Since
$c+S\geq 0$, we can consider $(c+S)\dmu$ as a new measure. Therefore, for $p\geq 1$,  similar to Lemma \ref{lem23}, we can also deduce
\ba
\sum\limits_{k\in\Z}\int_{M}(c+S) f_{\rho,\,k} ^{p}\dmu \leq  \lf(\frac{ \rho-1 }{\rho}\rt)^{p-1}\int_{M}(c+S) f ^p\dmu.\no
\ea
Obviously, there holds
\ba
\sum\limits_{k\in\Z}\int_{M}|\n f_{\rho,\,k}|^{p}\dmu
&=&\sum\limits_{k\in\Z}\int_{\rho^{k}\leq  f \leq \rho^{k+1}}|\n f  |^{p}\dmu\no\\
&=& \int_{M}|\n f  |^{p}\dmu.\no
\ea
Therefore, for $\al\geq p$, we can get
\ba
\lf(\sum\limits_{k\in\Z}W(f_{\rho,k})^{\al}\rt)^{\frac{1}{\al}}
&=&   \lf(\sum\limits_{k\in\Z}\lf(\int_{M}\lf(|\n f_{\rho,\,k}|^p+S f_{\rho,\,k} ^{p}\rt)\md \mu+c \int_{M} f_{\rho,\,k} ^{p}\md \mu\rt)^{\frac{\al}{p}}\rt)^{\frac{1}{\al}}\no\\
&\leq&\lf( \sum\limits_{k\in\Z}\int_{M} |\n f_{\rho,\,k}|^p\md \mu
+ \sum\limits_{k\in\Z}\int_{M} (c+ S)f_{\rho,\,k} ^{p}\md \mu
\rt)^{\frac{1}{p}}\no\\
&\leq &\lf(  \int_{M} |\n f |^p\md \mu
+ \lf(\frac{ \rho-1 }{\rho}\rt)^{p-1}\int_{M} (c+ S)f  ^{p}\md \mu
\rt)^{\frac{1}{p}}\no\\
&\leq &W(f).\no
\ea
\end{proof}
\section{Preliminaries of geometric flow}\label{pre}
In this section, we give some fundamental properties about the geometric flow (\ref{grf}).
Let $(M,\,g)$ be an $n$-dimensional compact Riemannian manifold.
Motivated by \cite{perelman}, fixing a real-valued function $S\in C^{\infty}(M,\R)$, we can define, for any $h\in C^{\infty}(M,\R)$ with $\int_{M}e^{-h}\dmu=1$,
\ban
\F(g,h)=\int_{M}(S+|\n h|^2)e^{-h}\dmu\no
\ean
and
\begin{align}\label{entropy}
\W(g,f,\tau)=\int_{M}\lf[\tau(S+|\n f|^2)+f-n\rt]\frac{e^{-f}}{(4\pi \tau)^{\frac{n}{2}}}\dmu,
\end{align}
where $\tau$ is a positive number and $f\in C^{\infty}(M, \mathbb{R})$ satisfies
\be\label{vol-1}
\int_M \frac{e^{-f}}{(4\pi\tau)^{\frac{n}{2}}}\dmu=1.
\ee
Let  $v=e^{-\frac{h}{2}}$ and
\ba\label{u}
u=\frac{e^{-\frac{f}{2}}}{(4\pi\tau)^{\frac{n}{4}}}.
\ea
Then we have
$$
\F(g,h)=\F^{\ast}(g,v)=\intm(4|\n v|^2+Sv^2)\dmu, \;\intm v^2\dmu=1.
$$
and
\ba\label{wwstar}
{\mathcal W}(g, f, \tau)={\mathcal W}^*(g, u, \tau) -\frac{n}{2}\ln \tau-\frac{n}{2}\ln(4\pi)-n
\ea
where
$$
{\mathcal W}^*(g, u,\tau)=\int_M \left[\tau(4|\nabla u|^2+S u^2)-u^2 \ln u^2 \right] \dmu, \;\intm u^2\dmu=1.
$$
We define
$$
4\ld_{0}(g):=\inf\limits_{\intm v^2\dmu=1}{\mathcal F}^*(g, v)
$$
and
$$
\mu^*(g, \tau):=\inf\limits_{\intm u^2\dmu=1}{\mathcal W}^*(g, u, \tau).
$$
In the case of geometric flow (\ref{grf}), we take the function $S$ as $S(x,t)$, the trace of time-dependent symmetric $2$-tensor $\mS$ with respect to Riemannian metric $g(x,t)$.
\begin{lem}\label{dandiao}
Assume that $g(x,t)$ is a smooth solution to the geometric flow (\ref{grf}) in $M\times[0,T)$.
Let $h$ be a positive solution to the backward heat equation
$$
\frac{\partial}{\partial t} h(x,\,t)=-\Delta_{g(x,\,t)} h+|\nabla h|_{g(x,\,t)}^{2}-S(x,\,t).
$$
Then we have
\ba
\frac{\md \F}{\md t}
&=&\int_{M}\bigg(2|h_{ij}+\mS_{ij}|^2  +\mathcal {D}_2( \mathcal{S},\nabla h)\bigg)e^{-h}\dmut\no
\ea
and
\ba\label{wmono}
\frac{\md \W}{\md t}=
 \int_{M} \tau\lf[ 2
 \lf|f_{ij}+\mS_{ij}-\frac{1}{2\tau}g_{ij}\rt|^{2}
 +\mathcal {D}_2( \mathcal{S},\nabla f)
\rt]\frac{e^{-f}}{(4\pi \tau)^{ \frac{n}{2}}}\dmut,
\ea
where
\ba \label{nonlinearconjugate}
\frac{\partial}{\partial t}f(x,\,t)=-\Delta_{g(x,\,t)} f(x,\,t)+|\nabla f|_{g(x,\,t)}^2-S(x,\,t)+\frac{n}{2\tau(t)}
\ea
and for any $\sigma>0$ and $0\leq t^{\ast}<T$,
$$
\tau(t)=t^{\ast}+\sigma-t.
$$
In particular, both $\F$ entropy and $\W$ entropy are non-decreasing in $t$ if $\mathcal {D}_2( \mathcal{S},\cdot)$ is nonnegative  and all times $t\in[0,T)$, from which we can get that
$\ld_{0}(g(t))$ is non-decreasing of $t$ and
\ba \label{mu-monotone}
\mu^*(g(t), \sigma) \ge \mu^*(g(0), t+\sigma)+ \frac{n}{2}\ln \frac{\sigma}{t+\sigma}
\ea
for all $ t \in [0, T)$ and $\sigma>0$ (the case $t=0$ is trivial).
\end{lem}
\begin{proof} The proof here is just direct computation. We use the method in \cite{fangzhu}. Set
\begin{eqnarray*}
P=2\Delta h-|\nabla h|^{2}+S
\end{eqnarray*}
By Lemma 2.1 in \cite{fangzhu}, let us take $\alpha=2,\beta=1,\lambda=0, a=1, b=d=0, c=-1$. Then we can get
\begin{align}\label{f3}
\frac{\partial P}{\partial t}=&-\Delta
P+2\nabla P\cdot\nabla h+2|h_{ij}+S_{ij}|^{2}+\frac{\partial S}{\partial t}-\Delta S-2|S_{ij}|^{2}\nonumber\\
&-2\nabla h\cdot\nabla S+4h_i\nabla_j \mS_{ij}-2\mS_{ij}h_i h_j+2R_{ij}h_i h_j\nonumber\\
=&-\Delta
P+2\nabla P\cdot\nabla h+2|h_{ij}+\mS_{ij}|^{2}+\mathcal {D}_2( \mathcal{S},\nabla h).
\end{align}
Combining (\ref{f3}) and the definition of $\F$ entropy, we derive
\begin{align*}
\frac{\md \F}{\md t}=&\frac{\md}{\md t}\int_M Pe^{-h}\dmut
=\int_M\lf(\frac{\partial P}{\partial t}-P\frac{\partial h}{\partial t}-PS\rt)e^{-h}\dmut\\
=&\int_M\bigg[-\Delta
P+2\nabla P\cdot\nabla h+2|h_{ij}+\mS_{ij}|^{2}+\mathcal {D}_2( \mathcal{S},\nabla h)\\
&\quad\quad\quad+P(\Delta h-|\nabla h|^{2}+S)-PS\bigg]e^{-h}\dmut\\
=&\int_M\bigg[-e^{h}\Delta (Pe^{-h})+2|h_{ij}+\mS_{ij}|^{2}+\mathcal {D}_2( \mathcal{S},\nabla h)\bigg]e^{-h}\dmut\\
=&\int_{M}\bigg(2|h_{ij}+\mS_{ij}|^2  +\mathcal {D}_2( \mathcal{S},\nabla h)\bigg)e^{-h}\dmut.
\end{align*}
Hence, it follows that $\F$ entropy is non-decreasing.

The monotonicity of $\W$ entropy had been  proved in Theorem 3.1 of \cite{huanghong} (see also \cite{fangzhu,guohongxin}).

Since $\mathcal {D}_2( \mathcal{S},\cdot)$ is nonnegative, from (\ref{wwstar}) and (\ref{wmono}), we have
$$
\frac{\md}{\md t} {\mathcal W}^*(g,u,\tau) \ge \frac{n}{2}\frac{\md}{\md t} \ln \tau,
$$
where
$$
u=u(t)=
\frac{e^{-\frac{f(t)}{2}}}{(4\pi\tau(t))^{\frac{n}{4}}},
$$
which satisfies the equation
$$
\frac{\partial u}{\partial t}=-\Delta u-\frac{|\nabla u|^2}{u}+\frac{S}{2}u.
$$
It follows that
$$
\mu^*(g(t_1), \tau(t_1)) \le \mu^*(g(t_2, \tau(t_2))+\frac{n}{2} \ln \frac{\tau(t_1)}{\tau(t_2)}.
$$
Choosing
$t_1=0$
and
$t_2=t^*$
we can obtain
\ba \label{mu-1}
\mu^*(g(0), t^*+\sigma) \le \mu^*(g(t^*), \sigma)+\frac{n}{2} \ln \frac{t^*+\sigma}{\sigma}.
\ea
Since $0<t^*<T$ is arbitrary,  (\ref{mu-1}) can be rewritten  as (\ref{mu-monotone}).

Similarly, we can get that
$\ld_{0}(g(t))$ is non-decreasing of $t$.
\end{proof}
\begin{rem}
The authors would like to thank Professor Hong Huang for pointing out the references \cite{guohongxin,huanghong}.
\end{rem}
\begin{lem}\label{sjie}
Assume that $g(x,t)$ is a smooth solution to the geometric flow (\ref{grf}) in $M\times[0,T)$ and that $\mathcal {D}_2( \mathcal{S},\cdot)$ is nonnegative. We have
\ba\label{parabolic}
\min\limits_{x\in M}S(x,t)\geq \min\limits_{x\in M}S(x,0).
\ea
Moreover, we have either
\be\label{s1}
 S(x,t) \geq 0,
\ee
or
\be\label{s2}
 \min\limits_{x\in M}  S(x,t)\geq \frac{1}{\frac{1}{\min\limits_{x\in M}S(x,0)}-\frac{2t}{n}}.
\ee
\end{lem}
\begin{proof}
Since $\mathrm{D}_{2}(\mS,\cdot)$ is nonnegative, taking $X=0$, we have
$$
\frac{\partial S}{\partial
t}-\Delta S-2|\mathcal{S}_{ij}|_{g(t)}^2\geq0,
$$
from which we can get
$$
\frac{\partial S}{\partial
t}-\Delta S-\frac{2}{n}S^2\geq0.
$$
From the maximum principle, we have (\ref{parabolic}).

If $\min\limits_{x\in M}S(x,0)\geq0$, we have (\ref{s1}). Otherwise, at the minimal point of $S(x,t)$, we have
\ba\label{s2yong}
\frac{\md}{\md t} \lf(\min\limits_{x\in M}  S(x,t)\rt)-\frac{2}{n}\lf[\min\limits_{x\in M}  S(x,t)\rt]^2\geq 0.
\ea
From the theory of ordinary differential equation,  by (\ref{s2yong}), we can get (\ref{s2}).
\end{proof}
\section{Proofs of theorems about (logarithmic) Sobolev inequalities}\label{logsob}
We will also need the following elementary lemma (See for example \cite{ye1}).
\begin{lem} \label{stronglemma} Let $a>0$ and $b$ be constants. Then the minimum of the
function $y=a \sigma-\frac{n}{2} \ln \sigma +b$ for $\sigma>0$ is
$\frac{n}{2} \ln(\alpha a)$, where
\ba \label{strong}
\alpha=\frac{2e}{n} e^{\frac{2b}{n}}.
\ea
\end{lem}
\begin{proof}[Proof of Theorem \ref{thmlogsob1}]
For
$
u \in W^{1,2}(M)
$
with
$
\int_{M}u^2\dmuo=1,
$
taking
$$
\alpha=\frac{8(t+\sigma)}{n  C_S(M,g_0)^2},\quad S=S_0
$$ in  (\ref{logsobr1}), we have
\ba
\int_M u^2 \ln u^2\dmuo
&\le& (t+\sigma) \int_M \lf(4|\nabla u|_{0}^2+S_{0}u^2\rt)\dmuo\no\\
&&+\frac{n}{2}(2\ln   C_S(M,\,g_{0})+\ln n -2\ln2 -1)\nonumber \\
&&-\frac{n}{2} \ln (t+\sigma)+(t+\sigma)\lf(\frac{4}{   C_S(M,\,g_{0})^2 \mathrm{Vol}_{g_0}(M)^{\frac{2}{n}}}-\min S_{0}\rt).\no
\ea
It follows that
\ba\label{muo}
\mu^*(g(0), t+\sigma)
&\ge&
\frac{n}{2} \ln (t+\sigma) -(t+\sigma)\lf(\frac{4}{   C_S(M,\,g_{0})^2 \mathrm{Vol}_{g_0}(M)^{\frac{2}{n}}}-\min S_{0}\rt)\nonumber \\
&&-\frac{n}{2}(2\ln   C_S(M,\,g_{0})+\ln n -2\ln2 -1).
\ea
From (\ref{mu-monotone}) and (\ref{muo}), we can deduce
\ba
\mu^*(g(t), \sigma) &\ge& \frac{n}{2} \ln \sigma  -(t+\sigma)\lf(\frac{4}{   C_S(M,\,g_{0})^2 \mathrm{Vol}_{g_0}(M)^{\frac{2}{n}}}-\min S_{0}\rt)\nonumber \\
&&-\frac{n}{2}(2\ln   C_S(M,\,g_{0})+\ln n -2\ln2 -1),\no
\ea
or
\ba
\mu^*\lf(g(t), \frac{\sigma}{4}\rt)
&\ge&
\frac{n}{2} \ln \sigma
-\lf(t+\frac{\sigma}{4}\rt)\lf(\frac{4}{   C_S(M,\,g_{0})^2 \mathrm{Vol}_{g_0}(M)^{\frac{2}{n}}}-\min S_{0}\rt)
\nonumber \\
&&-\frac{n}{2}(2\ln   C_S(M,\,g_{0})+\ln n-1)\no
\ea
which is equivalent to (\ref{logsob1}).

Taking
$$
a=\int_M\lf(|\nabla u|_{t}^2+\frac{S_{t}}{4}u^2\rt)\dmut+\frac{A_1}{4}>0
$$
and
$b=A_1t+A_2$ in Lemma \ref{stronglemma}, from (\ref{logsob1}), we can get (\ref{stronglogsob1}).
\end{proof}
Before prove Theorem \ref{thmlogsob3}, we need the following lemma.
\begin{lem}\label{thmlogsob2}
Assume that $g(x,t)$ is a smooth solution to the geometric flow (\ref{grf}) in $M\times[0,T)$ and that $\mathcal{D}_{2}(\mS,\cdot)$ defined in (\ref{con}) is nonnegative. If $\ld_{0}(g_{0}) $ is positive, then  for any  $\sigma>0$ and $t \in [0, T)$   satisfying $t+\sigma \ge \frac{n}{8}C_S(M,g_0)^2\delta_0$, there holds
\ba
\int_M u^2 \ln u^2 \dmut &\le& \sigma \int_M \lf(|\nabla u|_{t}^2+\frac{S_{t}}{4}u^2\rt)\dmut -\frac{n}{2}\ln \sigma
\nonumber \\
&&+\frac{n}{2}\ln n + n \ln C_S(M,g_0)+\sigma_0(g_0)\label{sobolev2}
\ea
for any $u\in W^{1,2}(M)$ with  $\int_M u^2 \dmut=1$,
where  $C_S(M,g_0)$ is the Sobolev constant defined in (\ref{sob}), $\delta_0=\delta_0(g_0)$ is the number defined in
(\ref{delta-0}) and the number
$\sigma_0(g_0)$ is defined in (\ref{sigma-0}).
%
\end{lem}
\begin{proof}
Assume $t+\sigma \ge \frac{n}{8}C_S(M,g_0)^2\delta_0(g_0)$. Choosing
$$
A=\frac{8(t+\sigma)}{nC_S(M,g_0)^2}\geq \delta_0(g_0),
$$
from (\ref{RLS4}), we can deduce
\ba
\int_M u^2 \ln u^2\dmuo &\le& 4(t+\sigma)\int_M \lf(|\nabla u|_{0}^2+\frac{S_{0}}{4}u^2\rt)\dmuo
-\frac{n}{2}\ln (t+\sigma)\nonumber \\
&&+\frac{n}{2}\bigg(2\ln C_S(M,g_0)+\ln n-2\ln 2\bigg)+\sigma_0(g_0),\no
\ea
where
$u \in W^{1,2}(M)$ satisfying
$\int_M u^2\dmuo=1$.

It follows that
\ba\label{muo1}
\mu^*(g_0, t+\sigma) \ge \frac{n}{2}\ln (t+\sigma)-\frac{n}{2}\bigg(2\ln C_S(M,g_0)+\ln n-2\ln 2\bigg)-\sigma_0(g_0).
\ea
From (\ref{mu-monotone}) and (\ref{muo1}), we can deduce
$$
\mu^*(g(t), \sigma)\ge \frac{n}{2} \ln \sigma -\frac{n}{2}\bigg(2\ln C_S(M,g_0)+\ln n-2\ln 2\bigg)-\sigma_0(g_0)
$$
or
$$
\mu^*\lf(g(t), \frac{\sigma}{4}\rt)\ge \frac{n}{2} \ln \sigma -\frac{n}{2}\bigg(2\ln C_S(M,g_0)+\ln n \bigg)-\sigma_0(g_0),
$$
which is equivalent to (\ref{sobolev2}).
%
\end{proof}
\begin{rem}
In the case of Ricci flow (\ref{rf}),  the result in
Lemma \ref{thmlogsob2}  can be found in  Ye \cite{ye1}.
\end{rem}
Note that the proofs of Theorem \ref{thmlogsob1} and Lemma \ref{thmlogsob2} lead to the following general result. Indeed, Theorem \ref{thmlogsob1} and Lemma \ref{thmlogsob2} can be seen as its special examples.
\begin{thm}
Let $g(t)$ be a smooth solution of the geometric flow (\ref{grf}) on $M\times [0,\,T)$ for some (finite or infinite) $T>0$ with $\mathcal{D}_{2}(\mS,\cdot)$ defined in (\ref{con})  nonnegative and let $h(\sigma)$ be
a scalar function for $\sigma>0$.  Assume that
the initial metric $g_0=g(0)$ satisfies the logarithmic Sobolev inequality
$$
\int_M u^2 \ln u^2 \dmuo \le \sigma \int_M \lf(|\nabla u|_{0}^2 + \frac{S_{0}}{4}u^2\rt)\dmuo
+h(\sigma)
$$
for each $\sigma>0$ and all $u\in W^{1,2}(M)$ with $\int_M u^2 \dmuo=1$. Then
there holds at each $t \in [0, T)$
$$
\int_M u^2 \ln u^2 \dmut \le \sigma \int_M \lf(|\nabla u|_{t}^2 + \frac{S_{t}}{4}u^2\rt)\dmut
+h(4t+\sigma)-\frac{n}{2}\ln\frac{\sigma}{4t+\sigma}
$$
for each $\sigma>0$ and all $u\in W^{1,2}(M)$ with $\int_M u^2 \dmut=1$.
\end{thm}
Given Theorem \ref{thmlogsob1} and Lemma \ref{thmlogsob2}, we can deduce Theorem \ref{thmlogsob3}.
\begin{proof}[Proof of Theorem \ref{thmlogsob3}]
Let $t\in [0, T)$ and $\sigma>0$. If $t+\sigma < \frac{n}{8}C_S(M, g_0)^2 \delta_0(g_0)$, we apply (\ref{logsob1}) in Theorem \ref{thmlogsob1} and bound $t+\frac{\sigma}{4}$ in (\ref{logsob1}) by $\frac{n}{8}C_S(M, g_0)^2 \delta_0(g_0)$.
Otherwise, we apply (\ref{sobolev2}) in Lemma \ref{thmlogsob2}. Then we can deduce (\ref{logsob3}).  Since the eigenvalue $\lambda_0(g(t))$ is non-decreasing and $\ld_0(g_0)>0$ we have
$\lambda_0(g(t))>0$ for all $t$. Therefore, we can deduce $\int_M \lf(|\nabla u|_t^2+\frac{S_t}{4}u^2\rt)\dmut>0$ for all $t$.  From Lemma \ref{stronglemma} by setting $a=\int_M \lf(|\nabla u|_t^2+\frac{S_t}{4}u^2\rt)\dmut$ and $b=C$, we can get (\ref{stronglogsob3}).
\end{proof}
Here we give a special conclusion of Theorem \ref{thmlogsob3}.
\begin{cor}\label{corlogsob}
Suppose that $g(t)$ is a smooth solution of the geometric flow (\ref{grf}) on $M\times [0,\,T)$ for some (finite or infinite) $T>0$ with $\mathcal{D}_{2}(\mS,\cdot)$ defined in (\ref{con})  nonnegative. If $\ld_{0}(g_0)$ is positive, then for $t\in[0,\,T)$, we have
\ba\label{fcorlogsob1}
\mathrm{Vol}_{g(t)}(M)\geq e^{ -C}
\ea
when $\hat{S}_t\leq 0$, and
\ba\label{fcorlogsob2}
\mathrm{Vol}_{g(t)}(M)\geq e^{-\frac{1}{4}-C}\hat{S}_t^{-\frac{n}{2}}
\ea
when $\hat{S}_t> 0$. Here $C$ is the constant in Theorem \ref{thmlogsob3} and $\hat{S}_t$ is the average of $S_t$
$$
\hat{S}_t=\frac{\int_M S_{t}  \dmut}{\mathrm{Vol}_{g(t)}(M)} .
$$
\end{cor}
\begin{proof}
Taking $u=\mathrm{Vol}_{g(t)}(M)^{-\frac{1}{2}}$ in (\ref{logsob3}), we get
$$
  \ln \frac{1}{\mathrm{Vol}_{g(t)}(M)} \le \frac{\sigma}{4}\hat{S}_t
-\frac{n}{2}\ln \sigma +C.
$$
If $\hat{S}_t \leq 0$, then taking $\sigma=1$, we get (\ref{fcorlogsob1}). If $ \hat{S}_t > 0$, then taking $\sigma=\hat{S}_t^{-1}$, we get (\ref{fcorlogsob2}).
\end{proof}
\begin{rem}
In the case of Ricci flow (\ref{rf}),  the result in Corollary \ref{corlogsob} specializes to the one in  Ye \cite{ye1}.
\end{rem}
Given the logarithmic Sobolev inequalities in Theorem \ref{thmlogsob1} and Theorem \ref{thmlogsob3}, we can deduce the uniform Sobolev inequality along geometric flow (\ref{grf}).
\begin{proof}[Proof of Theorem \ref{thmsob}]
In the case $\ld_{0}(g_0)>0$, letting
$$
f=\frac{u}{\lf(\intm u^2\dmut\rt)^{\frac{1}{2}}},
$$
from (\ref{stronglogsob3}), we have
\ba
&&\intm\lf[u^2\ln\lf(\frac{u^2}{\intm u^2\dmut}\rt) \dmut\rt]\no\\
&\leq& n\lf(\intm u^2\dmut\rt)\ln\lf(\frac{\alpha_{II}\int_M \lf(|\nabla u|_{t}^2 +\frac{S_{t}}{4} u^2\rt)\dmut}{ \intm u^2 }\rt)^{\frac{1}{2}}\no\\
&\leq& n\lf(\intm u^2\dmut\rt)\ln\lf(\frac{\alpha_{II}\int_M \lf(|\nabla u|_{t}^2 +\lf(\frac{S_{t}-S_0^-}{4}\rt) u^2\rt)\dmut}{ \intm u^2\dmut }\rt)^{\frac{1}{2}} ,\no
\ea
where $S_0^{-}=\min\{0,S_0\}$.
Define
$$
W(f):=\lf\{\int_M \lf[|\nabla u|_{t}^2 +\lf(\frac{S_{t}-S_0^-}{4} \rt) u^2\rt]\dmut\rt\}^{\frac{1}{2}}.
$$
Then from Lemma \ref{lem6} and Lemma \ref{lem7}(by taking $\rho=2,\,p=2,\,s=1,\,q=\frac{2n}{n-2}$), we have
\ba\label{case1sob}
\lf(\intm u^{\frac{2n}{n-2}}\dmut\rt)^{\frac{n-2}{2n}} \leq \lf(2^{\frac{2n}{n-2}}-1\rt)^{\frac{n-2}{2n}}2^{\frac{n+2}{n-2}} \al_{II} W(f) .
\ea
Since
\ba
\intm  u^2\dmut
&=&\frac{\ld_{0}(g(t))}{\ld_{0}(g(t))}\intm  u^2\dmut\no\\
&\leq&\frac{1}{\ld_{0}(g(t))}\intm \lf(|\n u|_t^2+\frac{S_t}{4}u^2\rt) \dmut\no\\
&\leq&\frac{1}{\ld_{0}(g_0)}\intm \lf(|\n u|_t^2+\frac{S_t}{4}u^2\rt) \dmut,\no
\ea
substituting the expression of $\al_{II}$ into (\ref{case1sob}), we have (\ref{sobp}),
where
\ba\label{sobpc}
A=(2^{\frac{2n}{n-2}}-1)^{\frac{n-2}{n}}2^{\frac{4n }{n-2}}  e^{2+\frac{4C}{n}} \frac{\ld_{0}(g_0)-S_0^-/4}{\ld_{0}(g_0)} .
\ea
In case $T<\infty$, define
$$
W(f):=\lf\{\int_M \lf(|\nabla u|_{t}^2 +\frac{S_{t}+A_1}{4} u^2\rt)\dmut \rt\}^{\frac{1}{2}}.
$$
Then from (\ref{stronglogsob1}), Lemma \ref{lem6} and Lemma \ref{lem7}(by taking $\rho=2,\,p=2,\,s=1,\,q=\frac{2n}{n-2}$), we have (\ref{sobi}), where
\ba
A&=&\lf(2^{\frac{2n}{n-2}}-1\rt)^{\frac{n-2}{n}}2^{\frac{4n }{n-2}}\frac{e^2}{n^2}e^{\frac{4(A_1T+A_2)}{n}}, \label{sobia}\\
B&=&\lf(2^{\frac{2n}{n-2}}-1\rt)^{\frac{n-2}{n}}2^{\frac{4n }{n-2}}\frac{A_1e^2}{4n^2}e^{\frac{4(A_1T+A_2)}{n}}.\label{sobib}
\ea
\end{proof}
\section{The $\kappa$-noncollapsing estimates under geometric flow}\label{noncollapsing}
In the case of Ricci flow (\ref{rf}), the $\kappa$-noncollapsing property, the volume ratio between a geodesic ball and Euclidean
ball with the same radius is bounded from below, is first proved by Perelman \cite {perelman} under the assumption that curvature is bounded along the Ricci flow.  Here we get the $\kappa$-noncollapsing estimates as follows.
\begin{thm}\label{thmnoncollapsing}
Assume that $g(x,t)$ is a smooth solution to the geometric flow (\ref{grf}) in $M\times[0,T)$ and $\mathcal{D}_{2}(\mS,\cdot)$ defined in (\ref{con}) is nonnegative. There hold
\begin{enumerate}
\item if $\ld_{0}(g_{0})>0$ and $S_t\leq \frac{1}{r^2}$ holds on a geodesic ball $B(x,r)$, where $r>0$, then  for $t\in [0,\,T)$, there holds
$$
\mathrm{Vol}_{g(t)}(B(x,r))\geq \lf(\frac{1}{2^{n+3}A}\rt)^{\frac{n}{2}} r^n,
$$
where $A$ is a positive constant defined in (\ref{sobpc}).
\item if $T<\infty$ and $S_t\leq \frac{1}{r ^2}$ holds on a geodesic ball $B(x,r)$ with $0<r\leq L$, then for $t\in [0,\,T)$, there holds
$$
\mathrm{Vol}_{g(t)}(B(x,r ))\geq \lf(\frac{1}{2^{n+3}A+2L^2B}\rt)^{\frac{n}{2}} r ^n,
$$
where $A$ and $B$ are defined in (\ref{sobia}) and (\ref{sobib}) respectively.
\end{enumerate}
\end{thm}
\begin{rem}
In the case of Ricci flow, this version of the $\kappa$-noncollapsing property can be found in \cite{ye1}, which is particularly powerful and flexible and has important applications to Poincar\'{e} conjecture and the geometrization conjecture \cite{p2}. (More meanings and applications of the $\kappa$-noncollapsing property can be found in \cite{ye1} and the references therein.)
\end{rem}
The proof of Theorem \ref{thmnoncollapsing} is a direct result of the following lemma.
\begin{lem}
Let $(M,\,g)$ be an $n$-dimensional ($n\geq 3$) compact Riemannian manifold and $\mS$ be any symmetric $2$-tensor with trace $S=\suml_{i,j=1}^ng^{ij}\mS_{ij}$. Assume that for any $u\in W^{1,2}(M)$, there holds the Sobolev inequality
$$
\lf(\intm |u|^{\frac{2n}{n-2}}\dmu \rt)^{\frac{n-2}{n}}\leq A\intm\lf( |\n u | ^2+\frac{S }{4}u^2\rt)\dmu +
B\intm u^2\dmu .
$$
If $S\leq \frac{1}{r^2}$ holds on a geodesic ball $B(x,r)$ with $0<r\leq L$, then there holds
$$
\mathrm{Vol}_{g }(B(x,r ))\geq \lf(\frac{1}{2^{n+3}A+2L^2B}\rt)^{\frac{n}{2}} r ^n.
$$
\end{lem}
\begin{proof}
The proof is very similar to the proof of Lemma 6.1 in \cite{ye1}. Here we omit it.
\end{proof}
\section{The $\kappa$-noninflated estimates under geometric flow}\label{sectionnoninflated}
Except for $\kappa$ non-collapsing property, the $\kappa$-noninflated property (the volume ratio between a geodesic ball and Euclidean
ball with the same radius is bounded from above) is also very useful (in the case of K\"{a}hler-Ricci flow, the importance of upper bound of volume can be found in \cite{sesum,chenwang1} and references therein).

To make the $\kappa$ non-inflated property clear, we give a definition as follows.
\begin{defn}
A smooth, compact, $n$-dimensional geometric flow (\ref{grf}) is called $\kappa$ non-inflated at the point $(x_{0},\,t_{0})$ under scale $\rho$
if the following statement holds.
\begin{enumerate}
\item the geometric flow is defined in the space time cube
$$
\bigg\{(x,\,t):\md(x,x_{0},t_{0})<r,\;t\in\,[t_{0}-r^2,\,t_{0}]\bigg\},
$$
\item
for some positive constant $\al$, $S(x,t)\leq \frac{\al}{t_{0}-t}$ for all $(x,t)$ in the above cube.
\end{enumerate}
Then there exists a positive constant $\kappa$, which may depend on $\al$ such that
$$
 \mathrm{Vol}_{g(t_{0})}(\mathrm{B}(x_{0},r,t_{0}))\leq \kappa r^n.
$$
\end{defn}
\begin{rem}
In the $\kappa$ non-collapsing property, the condition $S(x,\,t)\leq \frac{1}{r^2}$ on the $B(x,\,t)$ is included in the one $S(x,t)\leq \frac{\al}{t_{0}-t}$ of  the $\kappa$ non-inflated property in the same space time cube.

In the case of Ricci flow (\ref{rf}), our definition is the same as the one in Zhang \cite{zhangqi3}.
\end{rem}
\begin{thm}\label{nonflated}
Assume that $g(x,t)$ is a smooth solution to the geometric flow (\ref{grf}) in $M\times[0,T)$ and $\mathcal{D}_{2}(\mS,\cdot)$ defined in (\ref{con}) and $ Ric - \mathcal{S}$ are nonnegative.
For any $x_0\in M$, the geometric (\ref{grf}) is $\kappa$ non-inflated at $(x_0,\,t_0)$ under scale $\sqrt{t_0}$, where $\kappa$ defined in (\ref{kappadef}) depends only on $g_0,\,t_0$ and $\al$.
\end{thm}
\begin{rem}
The $\kappa$ non-inflated property in Theorem \ref{nonflated} specializes to the one in Zhang \cite{zhangqi3} in the case of Ricci flow (\ref{rf}).
\end{rem}
In order to prove the $\kappa$ non-inflated property of geometric flow (\ref{grf}), we need the lemmas as follows.

Let $g(x,t)$ be a solution to the geometric flow (\ref{grf}) on $M\times [0,\,T)$, where $M$ is  a compact manifold and let $\ell,\,t$ be two moments in time such that $0<\ell<t<T$, and $x,\,z\in M$. Let $G=G(z,\ell;x,t)$ be the fundamental solution of the conjugate heat equation
$$
\de_{\ell}f(z,\ell)+\Delta_{g(z,\ell)} f(z,\ell)-S(z,\ell) f(z,\ell)=0.
$$
along the geometric flow (\ref{grf}). Fixing $z,\ell$, we know that $G$, as a function of $x$  and $t$, is the fundamental solution of heat equation (see for example Lemma 26.3 of Chapter 26 in \cite{chow})
\be\label{heg}
\de_{t}h(x,t)-\Delta_{g(x,t)}h(x,t)=0.
\ee
\begin{lem}
Assume that $g(x,t)$ is a smooth solution to the geometric flow (\ref{grf}) in $M\times[0,T)$ and that $\mathcal {D}_2( \mathcal{S},\cdot)$ defined in (\ref{con}) is nonnegative. We have
\be\label{gshangjie1}
\int_{M} G(z,\ell;x,t)\md\mu(x,t)\leq1+C(1+t-\ell)^{\frac{n}{2}},
\ee
where $C$ only depends on $\min\limits_{x\in M} S(x,0) $. In particular,
$
C=0
$
when
$$
S(x,t)\geq \min\limits_{x\in M}  S(x,0)\geq0.
$$
\end{lem}
\begin{proof}
Since
\ba\label{kd}
&&\frac{\md}{\md t}\int_{M} G(z,\ell;x,t)\md\mu(x,t)\nonumber\\
&=&\int_{M}\bigg[\Delta_{x}G(z,\ell;x,t)-S(x,t) G(z,\ell;x,t)\bigg]\md\mu(x,t)\nonumber\\
&=&-\int_{M} S(x,t) G(z,\ell;x,t) \md\mu(x,t),
\ea
from (\ref{s1}), (\ref{s2}) and (\ref{kd}), we have either
$$
\frac{\md}{\md t}\int_{M} G(z,\ell;x,t)\md\mu(x,t)\leq 0
$$
or
$$
\frac{\md}{\md t}\int_{M} G(z,\ell;x,t)\md\mu(x,t)\leq
\frac{ \int_{M} G(z,\ell;x,t)\md\mu(x,t)}{-\frac{1}{\min\limits_{x\in M}  S(x,0)}+\frac{2t}{n}}.
$$
Finally, we can deduce (\ref{gshangjie1}).
\end{proof}
\begin{lem}
Assume that $g(x,t)$ is a smooth solution to the geometric flow (\ref{grf}) in $M\times[0,T)$ and that $\mathcal {D}_2( \mathcal{S},\cdot)$ defined in (\ref{con}) is nonnegative. We have
\ba\label{gshangjie2}
 G(z, \ell; x, t) \leq \frac{\exp[L(t)   -t\inf\limits_{y\in M}  S^{-}(y,0)]}{(4  (t-\ell))^{\frac{n}{2}}},
\ea
where $0<\ell<t$ and
$$
L(t)=2A_{1}t+A_{2},
$$
with $A_1,\,A_2$ the same as the ones defined in (\ref{a1}) and (\ref{a2}) up to adding   constants depending only on $n$.

Moreover, if $S(x,\,0)\geq 0$, we have
\ba\label{gshangjie2'}
 G(z, \ell; x, t) \leq \frac{e^C}{(4  (t-\ell))^{\frac{n}{2}}},
\ea
where $C$ is the same as the ones defined in (\ref{logsob3}) up to   adding   constants depending only on $n$.
\end{lem}
\begin{proof}
Let $f=f(x,\,t)$ be a positive solution to (\ref{heg}). Give $T_{0}>\ell$ and $t\in (\ell,\,T_{0})$, defining
$$
p(t)=\frac{T_{0}-\ell}{T_{0}-t},
$$
we have $p(\ell)=1$ and $ p(T_{0})=+\infty$.

Applying  the idea of Davies, we have
$$
\bag
\de_{t}\|f\|_{p(t)}
=&\de_{t}\lf[\lf(\int_{M}f^{p(t)} \md\mu(x,t)\rt)^{\frac{1}{p(t)}}\rt]\\
=&-\frac{p'(t)}{p^2(t)}\|f\|_{p(t)}\ln\int_{M}f^{p(t)} \md\mu(x,t)+\frac{1}{p(t)}\lf(\int_{M}f^{p(t)} \md\mu(x,t)\rt)^{\frac{1}{p(t)}-1}\\
&\times\lf[\int_{M}f^{p(t)}(\ln f)p'(t)\md\mu(x,t)\rt.\\
&\lf.\quad\quad+\int_{M}f^{p(t)-1}(p(t)\Delta_{x} f(x,t)-f(x,t)S(x,t))\md\mu(x,t)\rt]
\eag
$$
multiplying both sides by $p(t)^2\|f\|_{p(t)}^{p(t)-1}$, we can deduce
$$
\bag
p(t)^2\|f\|_{p(t)}^{p(t)-1}\de_{t}\|f\|_{p(t)}
=&- p'(t)\|f\|_{p(t)}^{p(t)} \ln\int_{M}f^{p(t)} \md\mu(x,t) \\
&+p(t)p'(t)   \int_{M}f^{p(t)} (\ln f)\md\mu(x,t) \\
&-4(p(t)-1)\int_{M}\lf|\n\lf (f^{\frac{p(t)}{2}}\rt)\rt|^2\md\mu(x,t) \\
&-p(t)\int_{M} \lf(f^{\frac{p(t)}{2}}\rt)^2S(x,t)\md\mu(x,t)
\eag
$$
Define $v(x,\,t)=\frac{f^{\frac{p(t)}{2}}}{\lf(\int_{M}f^{p(t)}\md\mu(x,t)\rt)^{\frac{1}{2}}}$. Then we have
$$
\bag
\|v\|_{2}=&1,\\
\int_{M}v^{2}\ln v^{2}
=&p(t)\int_{M}v^2\ln f-2\int_{M}v^2\ln\|f^{\frac{p(t)}{2}}\|_{2}\\
=&-2\ln \|f^{\frac{p(t)}{2}} \|_{2}+p(t)\int_{M}v^2\ln f.
\eag
$$
Dividing both sides by $\|f\|_{p(t)}^{p(t)}$, we have
\be
\bag
&p^{2}(t)\de_{t}\ln \|f\|_{p(t)}\\
=&p'(t)\int_{M}v^2\ln v^2\md\mu(x,t)-4(p(t)-1)\int_{M}|\n v|^2\md\mu(x,t)\\
&-p(t)\int_{M}S(x,t) v^2\md\mu(x,t)\\
=&p'(t)\int_{M}v^2\ln v^2\md\mu(x,t)
- \int_{M}S(x,t) v^2\md\mu(x,t)\\
&-4(p(t)-1)\int_{M}\lf(|\n v|^2+\frac{1}{4}S(x,t) v^2\rt)\md\mu(x,t)\\
\eag
\ee
From the Cauchy-Schwarz inequality, we have
$$
\bag
\frac{4(p(t)-1)}{p'(t)}
= &\frac{4(t-\ell)(T_{0}-t)}{T_{0}-\ell}\\
\leq&\frac{ (T_{0}-t+t-\ell)^2}{T_{0}-\ell}\\
=&T_{0}-\ell,\\
\frac{1}{p'(t)}
=&\frac{(T_{0}-t)^2}{T_{0}-\ell}\leq T_{0}-\ell.
\eag
$$
Therefore, we have
\be
\bag\label{inequalitysbig}
&p^{2}(t)\de_{t}\ln \|f\|_{p(t)}\\
=&
p'(t)\Bigg[
\int_{M}v^2\ln v^2\md\mu(x,t)- \frac{1}{p'(t)}\int_{M}S(x,t) v^2\md\mu(x,t)\\
&\quad\quad-\frac{4(p(t)-1)}{p'(t)}\int_{M}\lf(|\n v|^2+\frac{1}{4}S(x,t)  v^2\rt)\md\mu(x,t)
\Bigg]\\
\leq&
p'(t)\Bigg[
\int_{M}v^2\ln v^2\md\mu(x,t)-(T_0-\ell)\inf\limits_{x\in M}  S^{-}(x,t)  \\
&\quad\quad\quad-\frac{4(p(t)-1)}{p'(t)}\int_{M}\lf(|\n v|^2+\frac{1}{4}S(x,t) v^2\rt)  \md\mu(x,t)
\Bigg].
\eag
\ee
Taking
$$
\sigma=\frac{4(p(t)-1)}{p'(t)}\leq T_{0}-\ell
$$
in (\ref{logsob1}),  we can deduce
$$
p^{2}(t)\de_{t}\ln \|f\|_{p(t)}
\leq p'(t)\lf(-n\ln\sqrt{\frac{4(p(t)-1)}{p'(t)}}+L(T_0)-(T_0-\ell)\inf\limits_{x\in M}  S^{-}(x,0)\rt)
$$
where, since $\sigma\leq T_{0}-\ell\leq T_{0}$,
$$
A_{1}\lf(t+\frac{\sigma}{4}\rt)+A_{2}
 \leq A_{1}T_{0}+A_{2}
 = L(T_{0})
$$
and we also make use of   (\ref{parabolic})   to obtain
$$
-\inf\limits_{x\in M}  S^{-}(x,t) \leq-\inf\limits_{x\in M}  S^{-}(x,0).
$$

Since
$$
\frac{p'(t)}{p^2(t)}=\frac{1}{T_{0}-\ell}
$$
and
$$
\frac{4(p(t)-1)}{p'(t)}=\frac{4(t-\ell)[T_{0}-\ell-(t-\ell)]}{T_{0}-\ell},
$$
we can deduce
\ba
\de_{t}\ln \|f\|_{p(t)}
&\leq &\frac{1}{T_{0}-\ell}\Bigg\{-\frac{n}{2}\ln\lf[\frac{4(t-\ell)[T_{0}-\ell-(t-\ell)]}{T_{0}-\ell}\rt]\no\\
&&\quad\quad\quad\quad+L(T_{0})-(T_0-\ell)\inf\limits_{x\in M}  S^{-}(x,0)\Bigg\}.\no
\ea
Integrating from $t=\ell$ to $t=T_{0}$, we can get
$$
\ln \frac{\|f(\cdot,T_{0})\|_{\infty}}{\|f(\cdot,\ell)\|_{1}}\leq -\frac{n}{2}\ln[4(T_{0}-\ell)]+L(T_{0})-(T_0-\ell)\inf\limits_{x\in M}  S^{-}(x,0)+n.
$$
Since
 $$
 f(x, T_{0}) = \int_{  M}  G(z, \ell; x, T_{0}) f(z, \ell) \md\mu(z,\ell),
 $$
the above inequality implies that
$$
 G(z, \ell; x, T_{0}) \leq \frac{\exp[L(T_{0})   -(T_{0}-\ell)\inf\limits_{x\in M}  S^{-}(x,0)+n]}{(4  (T_{0}-\ell))^{\frac{n}{2}}}.
$$
Since $T_{0}>\ell$ is arbitrary, we get (\ref{gshangjie2}) with  maybe modified constants $A_1$ and $A_2$.

If $S(x,\,0)\geq 0$, then we can use the logarithmic Sobolev inequality (\ref{logsob3}) in (\ref{inequalitysbig}). Therefore, we can deduce (\ref{gshangjie2'}) with a modified constant.
\end{proof}
\begin{rem}
We can also prove this lemma by Moser's Iteration. Here we follow \cite{jiang} and just sketch it.

For $p\geq 1$, we have
$$
\intm f^pf_t\dmut-\intm f^p\Delta f\dmut=0,
$$
that is,
$$
\frac{1}{p+1}\de_t\intm f^{p+1}\dmut+\frac{1}{p+1} \intm S_t f^{p+1}\dmut +\frac{4p}{(p+1)^2}\intm |\n f^{\frac{p+1}{2}}|^2\dmut=0,
$$
where we use the Stokes' theorem and that
$
\de_t \dmut=-S_t\dmut.
$
Since $p\geq 1$, we have $4p\geq 2(p+1)$. Therefore, we can deduce
\ba\label{moser1}
\de_t\intm f^{p+1}\dmut&+&\intm \lf(S_t+C_0\rt)f^{p+1}\dmut+2\intm |\n f^{\frac{p+1}{2}}|^2\dmut\no\\
&\leq& C_0 \intm f^{p+1}\dmut,
\ea
where
$$
C_0=\lf\{
\bag
0,\quad \min_{x}S_0\geq 0,\\
\frac{4}{  C_S(M, g_0)^2\mathrm{Vol}_{g_0}(M)^{\frac{2}{n}}}-\min S_{0} ,\quad \min_{x}S_0< 0.
\eag\rt.
$$
Define
$$
\eta(t)=\lf\{\bag
&0,&\quad 0\leq t\leq \tau T,\\
&\frac{t-\tau T}{(\theta-\tau)T},&\quad \tau T\leq t\leq \theta T,\\
&1,&\quad\theta T\leq t\leq T.
\eag\rt.
$$
Multiplying (\ref{moser1}) by $\eta(t)$, we can deduce
\ba
\de_t\lf(\eta(t)\intm f^{p+1}\dmut\rt)&+&\frac{1}{2}\eta(t)\lf(\intm \lf(S_t+C_0 \rt)f^{p+1}\dmut+4\intm |\n f^{\frac{p+1}{2}}|^2\dmut\rt)\no\\
&\leq&\lf( C_0 +\eta'(t)\rt) \intm f^{p+1}\dmut.\no
\ea
Integrating this with respect to $t$ gives
\ba
\sup_{\theta T\leq t\leq T}\intm f^{p+1}\dmut &+&2 \lf\{\int_{\theta T}^{T}\intm \lf(|\n f^{\frac{p+1}{2}}|^2 +\frac{S_t+C_0}{4} f^{p+1}\rt)\dmut  \md t\rt\}\no\\
&\leq& 2\lf(\frac{1}{(\theta-\tau)T}+C_0\rt) \int_{\tau T}^{T}\intm f^{p+1}\dmut\md t.\no
\ea
From Lemma \ref{sjie}, we know that $S_t+C_0\geq 0$. From the proof of Theorem \ref{thmsob}, we can have the Sobolev inequality
\ba\label{case1sobrem}
\lf(\intm u^{\frac{2n}{n-2}}\dmut\rt)^{\frac{n-2}{ n}} \leq A  \int_M \lf[|\nabla u|_{t}^2 +\frac{S_t+C_0}{4} u^2\rt]\dmut,
\ea
where
\ba\label{sobpcrem}
A=\lf\{\bag
\lf(2^{\frac{2n}{n-2}}-1\rt)^{\frac{n-2}{n}}2^{\frac{4n }{n-2}}\lf[C_S(M,g_0)\rt]^4 e^{2+\frac{4}{n}\sigma_0(g_0)},& \quad  \inf_{x\in M}S_0 \geq 0,\\
\frac{1}{n^2}\lf(2^{\frac{2n}{n-2}}-1\rt)^{\frac{n-2}{n}}2^{\frac{4n }{n-2}}e^{2+\frac{4(A_1t+A_2)}{n}},& \quad \inf_{x\in M}S_0<0.
\eag\rt.
\ea
By making use of the Sobolev inequality above, we can get
\ba
 &&\int_{\theta T}^T\int_M f^{(p+1)\lf(1+\frac{2}{n}\rt)}\dmut\md t\no\\
&\le&\int_{\theta T}^T\left(\int_M f^{p+1}d\mu(t)\right)^{\frac{2}{n}}\left(\int_M f^{(p+1)\frac{n}{n-2}}\dmut\right)^{\frac{n-2}{n}}\md t\no\\
&\le& \sup_{\theta T \le t\le T}\left(\int_M f^{p+1}\dmut\right)^{\frac{2}{n}}A\int_{\theta T}^{T}\left(\int_M\lf[(S_t+C_0)f^{p+1}d\mu(t)+4|\nabla f^{\frac{p+1}{2}}|^2\rt]\dmut\right)\md t\no\\
&\le& 4A\lf[ C_0+\frac{1}{(\theta-\tau)T}\rt]^{1+\frac{2}{n}}\left(\int_{\tau T}^{T}\int_M f^{p+1}\dmut\md t\right)^{1+\frac{2}{n}}.\no
\ea
For $p\geq 2,\,0<\tau<1$, Set
$$
H(p,\,\tau):=\left(\int_{\tau T}^{T}\int_M f^{p}\dmut\md t\right)^{\frac{1}{p}},\quad \chi=\frac{n+2}{n}.
$$
Then for $0<\tau<\theta<1$, we have
\be\label{bfo16}
H(p\chi,\,\theta)\leq (4A)^{\frac{1}{p\chi}}\lf(C_0+\frac{1}{(\theta-\tau)T}\rt)^{\frac{1}{p}}H(p,\,\tau).
\ee
For $p_0\geq 2$ fixed, defining
$$
\bag
\gamma_{i}=&p_{0}\chi^{i-1},\quad \theta_{i}=\theta-\frac{\theta-\tau}{2^{i-1}} ,
\eag
$$
from (\ref{bfo16}), we have
\ba
H(\gamma_{k+1},\,\theta_{k+1})
&\leq&(4A)^{\frac{1}{p_0\chi^{k }}}\lf(C_0+\frac{2^k}{(\theta-\tau)T}\rt)^{\frac{1}{p_0\chi^{k-1}}}H(\gamma_{k},\,\theta_{k}).\no
\ea
By iteration, we can deduce
\ba
H(\gamma_{k+1},\,\theta_{k+1})
&\leq&(4A)^{\frac{1}{p_0}\sum\limits_{\ell=1}^k\frac{1}{\chi^{\ell }}}2^{\frac{1}{p_0}\sum\limits_{\ell=1}^k\frac{\ell}{\chi^{\ell-1}}}\no\\
&&\lf(C_0+\frac{1}{(\theta-\tau)T}\rt)^{\frac{1}{p_0}\sum\limits_{\ell=1}^k\frac{1}{\chi^{\ell-1}}}H(\gamma_{1},\,\theta_{1}).\no
\ea
Letting $k\longrightarrow +\infty$, we have
\ba
\sup_{(x,t)\in M\times [\theta T,\,T]}|f(x,\,t)|
&\leq&(4A)^{\frac{n+2}{2p_0}}2^{\frac{(n+2)^2}{4p_0}}\lf(C_0+\frac{1}{(\theta-\tau)T}\rt)^{\frac{n+2}{2p_0}}\no\\
&&\left(\int_{\tau T}^{T}\int_M f^{p_0}\dmut\md t\right)^{\frac{1}{p_0}}.\no
\ea
For $0<p<2$, we set
$$
h(\tau)=\sup_{(x,t)\in M\times [\tau T,\,T]}|f(x,\,t)| .
$$
Then from the Young's inequality, we can get
$$
h(\theta)\leq \frac{1}{2}h(\tau)+\frac{p}{2}\lf[2^{\frac{(n+2)^2}{8}}(2-p)\rt]^{\frac{2-p}{p}}
(4A)^{\frac{n+2}{2p }} \lf(C_0+\frac{1}{(\theta-\tau)T}\rt)^{\frac{n+2}{2p }}\left(\int_{\tau T}^{T}\int_M f^{p}\dmut\md t\right)^{\frac{1}{p}}
$$
Then from Lemma 4.3 in \cite{hanlin}, we get
\be\label{moser4}
h(\theta)\leq C
 A^{\frac{n+2}{2p }} \lf(C_0+\frac{1}{(\theta-\tau)T}\rt)^{\frac{n+2}{2p }}\left(\int_{\tau T}^{T}\int_M f^{p}\dmut\md t\right)^{\frac{1}{p}},
\ee
where $C$ is a constant depending only on $n$ and $p$.

Taking $p=1$ in (\ref{moser4}), from (\ref{gshangjie1}), we can get the estimates in the form of (\ref{gshangjie2}) and (\ref{gshangjie2'}).
\end{rem}
\begin{lem}
Assume that $g(x,t)$ is a smooth solution to the geometric flow (\ref{grf}) in $M\times[0,T)$ and that $\mathcal {D}_2( \mathcal{S},\cdot)$ defined in (\ref{con}) is nonnegative. For $0\leq\ell< t<T$ and any point $x$, we have
\ba\label{gxiajie1}
G(x,\ell;x,t)\geq \frac{1}{(4\pi(t-\ell))^{\frac{n}{2}}}e^{-\frac{1}{2\sqrt{t-\ell}}\int_{\ell}^{t}\sqrt{t-s}S(x,s)\md s}.
\ea
\end{lem}
\begin{proof}
For fixed $(x,t)$, consider $G(z,\ell;x,t)$ as a function of $(z,\ell),\;0\leq \ell< t$. Define $h(z,\ell)$ by
$$
G(z,\ell;x,t)=\frac{e^{-h(z,\ell)}}{(4\pi(t-\ell))^{\frac{n}{2}}}.
$$
Then we have
\ba\label{hequ}
\de_{\ell}h(z,\ell)+\Delta_{g(z,\ell)}h(z,\ell)-|\n h|_{g(z,\ell)}^2+S(z,\ell)-\frac{n}{2(t-\ell)}=0.
\ea
If $\mathcal {D}_2( \mathcal{S},\cdot)$ is nonnegative, Cao, Guo and Tran \cite{caoxiaodong} proved
\ba\label{cao}
(t-\ell)\bigg(2\Delta_{g(z,\ell)}h(z,\ell)-|\n h|_{g(z,\ell)}^2+S(z,\ell)\bigg)+h(z,\ell)-n\leq0.
\ea
From (\ref{hequ}) and (\ref{cao}), we have
$$
-\de_{\ell}h(z,\ell)\leq\frac{1}{2}S(z,\ell)-\frac{1}{2}|\n h|_{g(z,\ell)}^2-\frac{h(z,\ell)}{2(t-\ell)}.
$$
Thus, for any smooth
curve $\gamma(\ell)$, we have
\ba\label{diffha}
-\frac{\md}{\md \ell} h(\gamma(\ell),\ell)\leq\frac{1}{2}\bigg(S(\gamma(\ell),\ell)+|\dot{\gamma}(\ell)|_{g(\gamma(\ell),\ell)}^2\bigg)-\frac{h(\gamma(\ell),\ell)}{2(t-\ell)}.
\ea
Taking $\gamma(\ell)\equiv x$, integrating from $\ell=t_{2}$ to $\ell=t_1$, we have
$$
h(x,t_{2})\sqrt{t-t_2}\leq h(x,t_1)\sqrt{t-t_1}+\frac{1}{2}\int_{t_2}^{t_1}S(x,\ell)\sqrt{t-\ell}\md \ell,
$$
where $0\leq t_2< t_1\leq t$.

From Theorem 24.21 in \cite{chow}, we know that
$
\lim\limits_{t_1 \nearrow t}(t-t_1)^{\frac{n}{2}}G(x,t_1;x,t)
$
is bounded. Thus, for any $0\leq\ell< t$, we have
$$
h(x,\ell)\leq \frac{1}{2\sqrt{t-\ell}}\int_{\ell}^{t}\sqrt{t-s}S(x,s)\md s.
$$
Therefore, we can deduce (\ref{gxiajie1}).
\end{proof}
\begin{lem}\label{lemgxiajieyong}
Assume that $g(x,t)$ is a smooth solution to the geometric flow (\ref{grf}) in $M\times[0,T)$ and that $\mathcal {D}_2( \mathcal{S},\cdot)$ defined in (\ref{con}) and $ Ric - \mathcal{S}$ are nonnegative. Let
$u(x,t)$ be the positive solution of
$$
\frac{\de}{\de t}u(x,t)=\Delta_{g(x,t)}u(x,t).
$$
Then,
for $\delta>0$ and any  $x,\,y \in  M$, we have
\begin{equation}\label{PI2}
u(x,t)\leq U^{\frac{\delta}{1+\delta}}[u(y,t)]^{\frac{1}{1+\delta}}e^{\frac{\dist^2(x,y,t)}{4(t-s_0)\delta}},
\end{equation}
where $U=\sup\limits_{(x,s)\in M\times[s_0,t]}u(x,s) $.
\end{lem}
\begin{proof}
By Theorem 2.2 in \cite{fang}, we know that for any $0\leq s_0< t $ and $s\in[s_0,t ]$
\begin{equation}\label{PI1}
\frac{|\nabla u(x,s)|}{u(x,s)}\leq \sqrt{\frac{1}{s-s_0}}\sqrt{\ln \frac{U}{u(x,s)}}.
\end{equation}
Set
$
\psi(x,s)=\ln\frac{U}{u(x,s)},
$
then inequality (\ref{PI1}) yields
\[\lf|\nabla \sqrt{\psi(x,s)}\rt|=\frac{1}{2}\lf|\frac{\nabla u}{u \sqrt{\psi}}\rt|\leq\frac{1}{\sqrt{4(s-s_0)}}.\]
Next, for any $x,y\in M$, let $\gamma :[0,1]\rightarrow M$ be a minimizing geodesic such that $\gamma(0)=x$ and $\gamma(1)=y$. Integrating the above inequality along the geodesic, we get
$$
\sqrt{\ln\frac{U}{u(y,s)}}\leq\sqrt{\ln\frac{U}{u(x,s)}}+\frac{\dist(x,y,s)}{\sqrt{4(s-s_0)}},
$$
Thus, for any $\delta >0$, we have
\begin{eqnarray*}
\ln\frac{U}{u(y,s)}&\leq&\ln\frac{U}{u(x,s)}+\frac{\dist^2(x,y,s)}{4(t-s_0)}+\sqrt{\ln\frac{U}{u(x,s)}}\frac{\dist(x,y,s)}{\sqrt{s-s_0}}\nonumber\\
&\leq&\ln\frac{U}{u(x,s)}+\frac{\dist^2(x,y,s)}{4(s-s_0)}+\delta\ln\frac{U}{u(x,s)}+\frac{\dist^2(x,y,s)}{4(s-s_0)\delta}.
\end{eqnarray*}
Taking exponential of both sides in the above inequality and taking $s=t$, we gives (\ref{PI2}).
\end{proof}
\begin{lem}
Assume that $g(x,t)$ is a smooth solution to the geometric flow (\ref{grf}) in $M\times[0,T)$ and that $\mathcal {D}_2( \mathcal{S},\cdot)$ defined in (\ref{con}) and $ Ric - \mathcal{S}$ are nonnegative. We have
\ba\label{gxiajie2}
G(z,\ell;y,t)\geq \frac{c_1J(t)}{(t-\ell)^{\frac{n}{2}}}e^{-\frac{\dist^2(z,y,t)}{t-\ell}}e^{-\frac{1}{\sqrt{t-\ell}}\int_{\ell}^{t}\sqrt{t-s}S(z,s)\md s},
\ea
where $c_1$ depends only on $n$.
\end{lem}
\begin{proof} Set
$$
u(x,t)=G(z,\ell;x,t),\quad s_0=\frac{\ell+t}{2},\quad K=\sup\limits_{M\times[\frac{t+\ell}{2},\,t]}G(z,\ell;\cdot,\cdot),
$$
From Lemma \ref{lemgxiajieyong}, we have
\ba\label{351}
G(z,\ell;z,t)
&\leq& K^{\frac{\delta}{1+\delta}}[G(z,\ell;y,t)]^{\frac{1}{1+\delta}}e^{\frac{\dist^2(z,y,t)}{2(t-\ell)\delta}}.
\ea
Using (\ref{gshangjie2}), we know that
$$
K\leq  \frac{\exp[L(t)   -(t-\ell)\inf\limits_{y\in M}  S^{-}(y,0)]}{(4  (t-\ell))^{\frac{n}{2}}}.
$$
Denote $\exp[-L(t)   +(t-\ell)\inf\limits_{y\in M}  S^{-}(y,0)]=J(t)$. Then taking $\delta=1$ in (\ref{351}), from (\ref{gxiajie1}), we have (\ref{gxiajie2}).
\end{proof}
\begin{proof}[Proof of Theorem \ref{nonflated}]
Picking any $r\in (0,\,\sqrt{t_{0}})$, we consider geometric flow (\ref{grf}) in the space time cube
$$
Q(x_{0},t_{0},r)=\lf\{(x,s)|\dist(x,x_{0},t_{0})<r,\quad s\in[t_{0}-r^2,\,t_{0}]\rt\}.
$$
For $r \in(0, \sqrt{t_0})$ and $x\in M$ with  $\dist(x_0, x, t_0) \le r$, from (\ref{gxiajie2}), we have
\ba\label{gxiajie3}
G(x_0, t_0-r^2; x, t_0)
&\ge& \frac{c_1 J(t_0)}{r^n}
e^{-1 }  e^{- \frac{1}{r}
\int^{t_0}_{t_0 - r^2} \sqrt{t_0-s}
 S(x_0, s) \md s}\no\\
 &\ge& \frac{c_1 J(t_0)}{r^n}
e^{-1 }  e^{- \frac{1}{r}
\int^{t_0}_{t_0 - r^2} \sqrt{t_0-s}
 \frac{\alpha}{t_0-s} \md s}\no\\
&=&\frac{c_1 J(t_0)}{r^n} e^{-1 - 2
\alpha}.
\ea
From (\ref{gshangjie1}) and (\ref{gxiajie3}), we deduce
\ba
1+C (1+r^2)^{\frac{n}{2}}
&\ge& \int_{M} G(x_0, t_0-r^2; x, t_0) \md\mu(x,t_0) \no\\
&\ge& \int_{\dist(x_0, x, t_0) \le r} G(x_0, t_0-r^2; x, t_0) \md\mu(x,t_{0}) \no\\
&\ge& \frac{c_1 J(t_0)}{r^n} e^{-1 - 2 \alpha} \int_{\dist(x_0, x,t_0) \le r}  \md\mu(x,t_0).\no
\ea
This implies
$$
\mathrm{Vol}_{g(t_0)}(B(x_0, r, t_0)) r^{-n} \le \frac{\lf[1+ C (1+t_0)^{\frac{n}{2}}\rt] e^{1 + 2 \alpha}}{cJ(t_0)}.
$$
Taking
\ba\label{kappadef}
\kappa=\frac{\lf[1+ C (1+t_0)^{\frac{n}{2}}\rt] e^{1 + 2 \alpha}}{cJ(t_0)},
\ea
we
obtain
$$
\mathrm{Vol}_{g(t_0)}(B(x_0, r, t_0)) \le \kappa   r^n.
$$
\end{proof}
\section{Applications}\label{app}
In this section, we will give some examples of the geometric flow (\ref{grf}).
First, we will consider the Lorentzian mean curvature flow(see \cite {holder,muller} and references therein).

Let $M^n$ be a closed $n$-dimensional spacelike hypersurface in an ambient Lorentzian manifold $L^{n+1}$ and let $F_0:\,M^n\longrightarrow L^{n+1}$ be a smooth immersion of $M^n$ into $L^{n+1}$. Consider a smooth one parameter family of immersions
$$
F(\cdot,\,t):\;M^n\longrightarrow L^{n+1}
$$
satisfying $F(\cdot,0)=F_0(\cdot)$ and
$$
\frac{\de F(p,\,t)}{\de t}=H(p,\,t)\nu(p,\,t),\quad \forall\;(p,\,t)\in\,M\times[0,\,T),
$$
where $H(p,\,t)$ and $\nu(p,\,t)$ denote the mean curvature and the future-oriented timelike normal vector for the hypersurface $M_t=F(M^n,\,t)$
at $F(p,\,t)$, respectively. It is easy to see that the induced metric solves the equation
\ba\label{mcf}
\frac{\de}{\de t}g_{ij} =2HA_{ij},
\ea
where $A=(A_{ij})$ is the second fundamental form on $M_t$.
\begin{thm}\label{thmmcf}
Let $L^{n+1}$ be the ambient Lorentzian manifold with nonnegative sectional curvature. Then for evolution (\ref{mcf}), Theorem \ref{thmlogsob1},
Theorem \ref{thmlogsob3}, Theorem \ref{thmsob},
Lemma \ref{thmlogsob2}, Corollary \ref{corlogsob}, Theorem \ref{thmnoncollapsing} and Theorem \ref{nonflated} hold.
\end{thm}
\begin{proof}
In this setting, we have $\mS_{ij}=-HA_{ij}$ and $S=-H^2$. Marking the curvature with respect to the ambient Lorentzian manifold $L^{n+1}$
with a bar, we have the Gauss equation
$$R_{ij}=\overline{R}_{ij}-HA_{ij}+A_{i\ell}A_{\ell j}+\overline{R}_{i0j0},$$
the Codazzi equation
$$\nabla_iA_{jk}-\nabla_kA_{ij}=\overline{R}_{0jki},$$
and the evolution equation for the mean curvature
$$\frac{\partial H}{\partial t}=\Delta H-H(|A|^2+\overline{Ric}(\nu,\nu)),$$
where $\nu$ denotes the future-oriented timelike normal vector, represented by 0 in the index-notation. Using the three identities above, we get
$$\mathcal {D}_2( \mathcal{S},X)= 2|\nabla H-A(X,\cdot)|^2+2\overline{Ric}(H\nu-X,H\nu-X)+2\langle \overline{Rm}(X,\nu)\nu,X\rangle.$$
Since the ambient Lorentzian manifold $L^{n+1}$ has nonnegative sectional curvature,  the nonnegativity constraints of
$\mathcal {D}_2( \mathcal{S},X)$ holds naturally.

We also have
$$
Ric(X,\,X)-\mS(X,\,X)
 = \overline{Ric}(X,\,X) +X^iA_{i\ell}A_{\ell j}X^j+
 \langle\overline{Rm}(X,\nu)\nu,X\rangle\geq 0.
$$
This completes the proof of Theorem \ref{thmmcf}.
\end{proof}
Second, let $M$ be a real $n$($=2m$) dimensional Fano manifold with K\"{a}hler form $\omega_0$ associated to the  K\"{a}hler metric   $g_0$. We consider the twisted K\"{a}hler-Ricci flow (see \cite{cabor,liu,zhangzhang} and the references therein)
\be
\lf\{\bag\label{tkrf}
\frac{\de }{\de t}g_{i\ov{j}}(x,t)=&-R_{i\ov{j}}(x,t)+\theta_{i\ov{j}}(x)+g_{i\ov{j}}(x,t),\\
g_{i\ov{j}}(x,0)=&(g_0)_{i\ov{j}}(x ),
\eag\rt.
\ee
where $\theta $ is a closed  semi-positive $(1,1)$ form and
$$
[2\pi c_1(M)]=[\omega(x,t)+\theta].
$$
Here $\omega(x,t)=\mn g_{i\ov{j}}(x,t)\md z^i\wedge\md \ov{z^j}$ is the K\"{a}hler form of $g(x,t)$. We have
\begin{thm}\label{thmtkrf}
Let $M$ be a real $n$($=2m$) dimensional Fano manifold with K\"{a}hler form $\omega_0$ whose K\"{a}hler metric is denoted by $g_0$.
Then for the twisted K\"{a}hler-Ricci flow (\ref{tkrf}) with the assumption above,  there exists a positive constant $\kappa>0$ depending only on the initial metric $g_0$ such that
$$
\mathrm{Vol}_{  g(t )}\bigg(B\lf(x,  r\rt)\bigg)\leq \kappa r^n, \quad\forall\;  (x,t)\in M\times (0,\,+\infty).
$$
\end{thm}
\begin{rem}
In the case of  K\"{a}hler-Ricci flow ($\theta_{i\ov{j}}\equiv 0$), the conclusion in Theorem \ref{thmtkrf} is the one in Zhang \cite{zhangqi3}(see also \cite{chenwang2}).
\end{rem}
\begin{rem}
From the scaling transformation (\ref{sag}), it is not difficult to know that Theorem \ref{thmlogsob1},
Theorem \ref{thmlogsob3}, Theorem \ref{thmsob}, Lemma \ref{thmlogsob2}, Corollary \ref{corlogsob}, Theorem \ref{thmnoncollapsing} and Theorem \ref{nonflated} also hold for twisted K\"{a}hler-Ricci flow (\ref{tkrf}).
\end{rem}
To avoid confusions, we give some preliminaries about K\"{a}hler geometry for special use in this paper.
Let $(M,\n,g)$ be real $n$-dimensional ($n=2m$) K\"{a}hler manifold, $\n$ be the Levi-Civita connection (also Chern connection) and $g$ be Riemannian metric which determines a unique K\"{a}hler metric and vise versa. So we can consider $g$ itself as the K\"{a}hler metric.
Assume that
$$
z=(z^1,\cdots,z^m)
$$
is the local coordinate system on $M$.
The K\"{a}her form is
$$
\omega=\mn \sum\limits_{i,j=1}^mg_{i\ov{j}}\md z^i\wedge\md \ov{z^j},
$$
where $g_{i\ov{j}}=g(\de_{z^i},\,\de_{\ov{z^j}})$.

Let $\theta$ be a real $(1,1)$-form. Then we have
$$
\ov{\theta_{i\ov{j}}}=\theta_{j\ov{i}},\;\tr_{g}\theta=2\sum\limits_{i,j=1}^mg^{\ov{j}i}\theta_{i\ov{j}},\;|\theta|_g^2=2\sum\limits_{i,j=1}^m g^{\ov{j}i}g^{\ov{q}p}\theta_{i\ov{q}}\theta_{p\ov{j}},
$$
where $\sum\limits_{ j=1}^mg^{\ov{j}i}g_{k\ov{j}}=\delta_{k}^i$.

If $\theta$ is also closed, then we have
$$
\frac{\theta_{i\ov{j}}}{\de z^k}=\frac{\theta_{k\ov{j}}}{\de z^i},\quad \frac{\theta_{i\ov{j}}}{\de \ov{z^{\ell}}}=\frac{\theta_{i\ov{\ell}}}{\de \ov{z^j}},
$$
which is equivalent to
$$
\n_{k}\theta_{i\ov{j}}=\n_{i}\theta_{k\ov{j}},\quad \n_{\ov{\ell}}\theta_{i\ov{j}}=\n_{\ov{j}}\theta_{i\ov{\ell}}.
$$
For any $f\in C^{\infty}(M,\R)$, we have
$$
\Delta_g f=2\sum\limits_{i,j=1}^mg^{\ov{j}i}\frac{\de^2 f}{\de z^i\de\ov{z^j}}.
$$
\begin{proof}[Proof of Theorem \ref{thmtkrf}]
For twisted K\"{a}hler-Ricci flow (\ref{tkrf}), define
$$
\mS_{i\ov{j}}(x,t)=R_{i\ov{j}}(x,t)-\theta_{i\ov{j}}(x).
$$
By making use of scaling
\be\label{sag}
t=-\ln(1-2s),\quad g_{i\ov{j}}(x,t)=\frac{1}{1-2s}\ti g_{i\ov{j}}(x,s),\quad s\in[0,\frac{1}{2}),
\ee
we know that $\ti g_{i\ov{j}}(x,s)$ satisfies the geometric flow equation
$$
\frac{\de }{\de s}\ti g_{i\ov{j}}(x,s)=-2\ti\mS_{i\ov{j}}(x,s),
$$
where $\ti g_{i\ov{j}}(x,0)=(g_0)_{i\ov{j}}(x)$ and
$$
\ti\mS_{i\ov{j}}(x,s)=R_{i\ov{j}}(x,-\ln(1-2s))-\theta_{i\ov{j}}(x).
$$
Then we can get
\ba\label{ti1}
\frac{\de \ti \mS}{\de s}-\Delta_{\ti g} \ti S-2|\ti\mS|_{\ti g}^2=0,
\ea
where $\ti S=\tr_{\ti g} \ti\mS$.

For any real-value vector $X\in \X$, it can be written as
$$
X=\sum_{i=1}^mX^i\de_{z^i}+\sum_{i=1}^m\ov{X^i\de_{z^i}}.
$$
Since $\theta$ is a real closed $(1,1)$-form, we have
\ba\label{ti2}
2\sum\limits_{i,j=1}^m\ti\nabla^i\theta_{i\ov{j}}\ov{X^j}
&=&2\sum\limits_{q,i,j=1}^m\ti g^{\ov{q}i}\ti\n_{\ov{q}}\theta_{i\ov{j}}\ov{X^j}\no\\
&=&2\sum\limits_{q,i,j=1}^m\ti g^{\ov{q}i}\ti\n_{\ov{j}}\theta_{i\ov{q}}\ov{X^j}\no\\
&=&\sum\limits_{ j=1}^m\lf(\ti\n_{\ov{j}}\tr_{\ti g}\theta\rt)\ov{X^j}.
\ea
From the second Bianchi identity, we also get
\be\label{ti3}
2\sum\limits_{i,j=1}^m\ti\nabla^i\ti R_{i\ov{j}}\ov{X^j}=\sum\limits_{ j=1}^m\lf(\ti\n_{\ov{j}} \ti R\rt)\ov{X^j}.
\ee
Since $\theta$ is semi-positive, from (\ref{ti1}), (\ref{ti2}) and (\ref{ti3}), we have
$$
\mathcal {D}_2(\ti \mS,X)=4\sum_{i,j=1}^m\theta_{i\ov{j}}X^i\ov{X^j}\geq 0.
$$
Collins and Sz\'{e}kelyhidi \cite{cabor} and Liu \cite{liu} proved that there exists a constant $\al>0$ such that
$$
\sum\limits_{i,j=1}^mg^{\ov{j}i}\lf(R_{i\ov{j}}(x,t)-\theta_{i\ov{j}}(x)\rt)\leq \al.
$$
Therefore, we have
$$
\ti S(x,s)\leq\frac{\al}{\frac{1}{2}- s},\quad s\in[0,\,\frac{1}{2}).
$$
Choose $s_0\in(0,\,\frac{1}{2})$ and $\ti r\in[0,\, \sqrt{s_0})$. Then for $s\in[s_0-\ti r^2,\,s_0]$ and $x\in M$, we have
$$
\ti S(x,s)\leq\frac{\al}{s_0- s}.
$$
By Theorem \ref{nonflated}, we have
\ba\label{mtnonflated}
\mathrm{Vol}_{\ti g(s_0)}(B(x,\ti r))\leq \kappa \ti r^n.
\ea
From (\ref{sag}), we know that
$$
\dist(x,y,\ti g(s))=\ti r
$$
implies
$$
\dist(x,y,  g(t))= r
$$
where
$$
t=-\ln(1-2s ),\quad r=\frac{\ti r}{\sqrt{1-2s}}.
$$
Therefore, from (\ref{mtnonflated}), we have
$$\label{nonflatedtgrf}
\mathrm{Vol}_{  g(t_0)}\lf[B\lf(x,  \frac{\ti r}{\sqrt{1-2s_0}}\rt)\rt]\leq \kappa \lf(\frac{\ti r}{\sqrt{1-2s_0}}\rt)^n,
$$
that is, at any point $(x,t)\in M\times (0,\,+\infty)$, for the twisted K\"{a}hler-Ricci flow (\ref{tkrf}), we have
\ba\label{nonflatedtgrf}
\mathrm{Vol}_{  g(t )}\lf[B\lf(x,  r\rt)\rt]\leq \kappa r^n,
\ea
where
$$
r\in\lf(0,\,\sqrt{\frac{e^t-1}{2}}\rt).
$$
Since Collins and Sz\'{e}kelyhidi \cite{cabor} and Liu \cite{liu} proved that the diameter of $(M,\,g(t))$ is uniformly bounded, the above  estimate (\ref{nonflatedtgrf}) holds for all $r>0$ with maybe a different constant $\kappa$.
\end{proof}
\noindent{\bf Acknowledgements}
This work was carried out while the authors were visiting Mathematics Department of Northwestern University. We would like to thank Professor Valentino Tosatti and Professor Ben Weinkove for hospitality and helpful discussions. The authors are also grateful to the anonymous referees and the editor for their careful reading and helpful suggestions which greatly improved the paper.

\begin{flushleft}
{\small Shouwen Fang\\
School of Mathematical Science, Yangzhou University,
Yangzhou, Jiangsu 225002, P. R. China\\
E-mail: shwfang@163.com\\
Tao Zheng\\
School of Mathematics and Statistics, Beijing Institute of Technology,
 Beijing 100081, P. R. China\\
E-mail: zhengtao08@amss.ac.cn}
\end{flushleft}


\begin{thebibliography}{10}
\bibitem{aubin}
Aubin, T.  \emph{Probl\`{e}mes isop\'{e}rim\'{e}triques et espaces de Sobolev}. (French) J. Differential Geom. \textbf{11} (1976), no.4, 573-598.

\bibitem{aubinli}
Aubin, T.; Li, Y.
\emph{On the best Sobolev inequality}, J. Math. Pures. Appl. (9) \textbf{78} (1999),  353-387.




\bibitem{MH}B\u{a}ile\c{s}teanu, M.; Tran, H. \emph{Heat kernel estimates under the Ricci-harmonic map flow}, 	arXiv:1310.1619.

\bibitem{bakry}
Bakry, D.; Coulhon, T.; Ledoux, M.; Saloff-Coste. L.  \emph{Sobolev inequalities in disguise}, Indiana Univ. Math. J. \textbf{44} (1995), 1033-1047.

\bibitem{biezuner}
Biezuner, R. J. \emph{Best constants, optimal sobolev inequalities on riemannian manifolds and applications},  Rutgers, The State University of New Jersey, 2003.


\bibitem{caoxiaodong}
Cao, X.; Guo, H.; Tran, H.  \emph{Harnack estimates for conjugate heat kernel on evolving manifolds}, Math. Z. \textbf{281} (2015), no. 1-2, 201-214.


\bibitem{carron}
Carron, G.  \emph{In\'{e}galit\'{e}s isop\'{e}rim\'{e}triques de Faber-Krahn et cons\'{e}quences}, Actes de la Table Ronde de G\'{e}om\'{e}trie Diff\'{e}rentielle (Luminy, 1992), 205-232, S\'{e}min. Congr., 1, Soc. Math. France, Paris, 1996.

\bibitem{chenwang1}
Chen, X.; Wang, B. \emph{Space of Ricci flow (I)},  Comm. Pure Appl. Math. \textbf{65} (2012), no. 10, 1399-1457.




\bibitem{chenwang2}
Chen, X.; Wang, B. \emph{On the conditions to extend Ricci flow (III)}, Int. Math. Res. Not. IMRN 2013, no. 10, 2349-2367.

\bibitem {chow}
Chow, B.; Chu, S.; Glickenstein, D.; Guenther, C.; Isenberg, J.; Ivey, T.; Knopf, D.; Lu, P.; Luo, F.; Ni, L.
 \emph{The Ricci flow: techniques and applications. Part {III}: Geometric-Analysis aspects},
Mathematical Surveys and Monographs Vol \textbf{163}. American Mathematical Society, Providence, RI.











\bibitem{cabor}
Collins, T.; Sz\'{e}kelyhidi, G.  \emph{The twisted K\"{a}hler-Ricci flow}, arXiv:1207.5441v1.



\bibitem{fang}Fang, S.  \emph{Differential Harnack estimates for  heat equations with potentials under geometric flows}, Arch. Math. \textbf{100} (2013), 179-189.

\bibitem{fangzheng}Fang, S.; Zheng, T.  \emph{An upper bound of the heat kernel along the harmonic-Ricci flow}, arXiv: 1501.00639.


\bibitem{fangzhu}Fang, S.; Zhu, P. \emph{Differential Harnack estimates for backward heat equations with potentials under geometric flows}, Commun. Pur. Appl. Anal. \textbf{14} (2015), 793-809.







\bibitem{gallot}Gallot, S. \emph{In\'{e}galit\'{e}s isop\'{e}rimetriques, courbure de Ricci et invariants g\'{e}om\'{e}triques. I},   C. R. Acad. Sci. Paris S\'{e}r. I Math. \textbf{296} (1983), 333-336.


\bibitem{gross}Gross, L. \emph{Logarithmic Sobolev inequalities}, Amer. J. Math. \textbf{97} (1975),1061-1083.


\bibitem{guohongxin}Guo, H.; Philipowski, R. Thalmaier, A. \emph{Entropy and lowest eigenvalue on evolving manifolds}, Pacific J. Math. \textbf{264} (2013), 61-81.



\bibitem{hanlin}
Han, Q.; Lin, F. \emph{Elliptic partial differential equations}, Courant Lect. Notes Math. \textbf{1}, Courant Institute of Mathematical Sciences, New York 1997.



\bibitem{hebey}
Hebey, E. \emph{Optimal Sobolev inequalities on complete Riemannian manifolds with Ricci curvature bounded below and positive injectivity radius}, Amer. J. Math. \textbf{118} (1996), no.2, 291-300.


\bibitem{hebey2}
Hebey, E. \emph{Nonlinear analysis on manifolds: Sobolev spaces and inequalities}, Courant Lecture Notes 1999.


\bibitem{hv}
Hebey, E.; Vaugon, M. \emph{Meilleures constantes dans le thorme d'inclusion de Sobolev}, (French) Ann. Inst. H. Poincar Anal. Non Linaire \textbf{13}
(1996), no. 1, 57-93.

\bibitem{holder}Holder, M.  \emph{Contracting spacelike hypersurfaces by their inverse mean curvature}, J. Austral. Math. Soc. Ser. A \textbf{68} (2000), no. 3, 285-300.

\bibitem{hsu}Hsu, S. \emph{Uniform Sobolev inequalities for manifolds evolving by Ricci flow}, arXiv: 0708.0803v1.



\bibitem{huanghong}Huang, H. \emph{Optimal transportation and monotonic quantities on evolving manifolds}, Pacific J. Math. \textbf{248} (2010), 305-316.




\bibitem{jiang}Jiang, W. \emph{Bergman Kernel along the K\"{a}hler-Ricci flow and Tian's conjecture}, J. rein angew. Math.,
http://dx.doi.org/10.1515/crelle-2014-0015, Ahead of print.





\bibitem{list}List, B. \emph{Evolution of an extended Ricci flow system}, Comm. Anal. Geom. \textbf{16} (2008), no. 5, 1007-1048.


\bibitem{liu}
Liu, J.  \emph{The generalized K\"{a}hler Ricci flow}, J. Math. Anal. Appl., \textbf{408} (2013) 751-761.


\bibitem{liuwang}Liu, X.; Wang, K. \emph{A Gaussian upper bound of the conjugate heat equation along an
extended Ricci flow}, 	arXiv:1412.3200.


\bibitem{muller}
M\"{u}ller, R.  \emph{Monotone volume formulas for geometric flows}, J. reine. angew. Math., \textbf{643} (2010) 39-57.

\bibitem{MR}M\"{u}ller, R. \emph{Ricci flow coupled with harmonic map flow},  Ann. Sci. \'{E}c. Norm. Sup\'{e}r. (4) \textbf{45} (2012), 101--142.

\bibitem{perelman}
Perelman, G. \emph{The entropy formula for the Ricci flow and its geometric applications}, arXiv:math/0211159.


\bibitem{p2}
Perelman, G. \emph{Ricci flow with surgery on three-manifolds},  arXiv:math/0303109.



\bibitem{saloff}
Saloff-Coste, L. \emph{Uniformly elliptic operators on Riemannian manifolds}, J. Differential Geom., \textbf{36} (1992), 417-450.


\bibitem{sesum}
\v{S}e\v{s}um, N.  \emph{Convergence of a K\"{a}hler-Ricci flow}, Math. Res. Lett. \textbf{12} (2005), 623-632.

\bibitem{sesumtian}
\v{S}e\v{s}um, N.; Tian, G. \emph{Bounding scalar curvature and diameter along the K\"{a}hler-Ricci flow (after Perelman)}, Journal of the Institute
of Mathematics of Jussieu, \textbf{7} (2008), 575-587.



\bibitem{talenti}
Talenti, G. \emph{Best constant in Sobolev inequality}, Ann. Mat. Pura Appl. \textbf{4} (1976), 353-372.

\bibitem{topping}
Topping, P. \emph{Lectures on the Ricci flow}. Vol. \textbf{325}. Cambridge University Press, 2006.


\bibitem{ye2}Ye, R.  \emph{The logarithmic Sobolev inequality along the Ricci flow in dimension $2$}, arXiv: 0708.2003.





\bibitem{ye4}Ye, R. \emph{Sobolve inequalities, Riesz transforms and the Ricci flow}, Commun. Math. Stat. \textbf{2} (2014), no. 2, 187-209.

\bibitem{ye3}Ye, R. \emph{The logarithmic Sobolev inequality along the Ricci flow: the case $\ld_0(g_0)=0$},  Commun. Math. Stat. \textbf{2} (2014), no. 3-4, 363-368.
\bibitem{ye1}Ye, R. \emph{The logarithmic Sobolev and Sobolev inequalities along the Ricci flow}, Commun. Math. Stat. \textbf{3} (2015), no. 1, 1-36.

\bibitem{zhangqi1} Zhang, Q. S. \emph{A uniform Sobolev inequality under Ricci flow}, Int. Math. Res. Not. IMRN \textbf{2007}, no. 17, Art. ID rnm056, 17pp. Erratum: Int. Math. Res. Not. IMRN \textbf{2007}, no. 19, Art. ID rnm096, 4 pp. Addendum: Int. Math. Res. Not. IMRN \textbf{2008}, no. 1, Art. ID rnm 138, 12 pp.


%


\bibitem{zhangqi2}Zhang, Q. S. \emph{Sobolev Inequalities, Heat Kernels under Ricci flow, and the Poincar\'{e} Conjecture}, CRC Press, Boca Raton, FL, 2011.



\bibitem{zhangqi3}
Zhang, Q. S. \emph{Bounds on volume growth of geodesic balls under Ricci flow}, Math. Res. Lett. \textbf{19} (2012), no. 1, 245-253.

\bibitem{zhangzhang}
Zhang, X.; Zhang, X. \emph{Generalized K\"{a}hler-Einstein metrics and energy functionals}, Canad. J. Math. \textbf{66} (2014), 1413-1435.

\bibitem{Zh}Zhu, A. \emph{Differential Harnack inequalities for the backward heat equation with potential
          under the harmonic-Ricci flow}, \textit{J. Math. Anal. Appl.} \textbf{406} (2013), 502--510.

\end{thebibliography}
\end{document}